\documentclass[12pt]{amsart}
\usepackage{amsmath}
\usepackage{amsthm}
\usepackage{amsfonts}
\usepackage{amssymb}
\usepackage{bm}
\usepackage{mathrsfs}
\usepackage{graphicx}
\usepackage{cprotect}
\usepackage{setspace}
\usepackage{amscd}
\usepackage{mathtools}
\usepackage{commath}
\usepackage{enumerate}
\usepackage{color}
\usepackage{hyperref}
\hypersetup{
    colorlinks=true,
    linkcolor=blue,
    filecolor=magenta,      
    urlcolor=cyan,
}
\usepackage{xcolor}
\usepackage{amsfonts}
\newcommand{\R}{\mathbb{R}} 
 
\newcommand{\C}{\mathbb{C}} 
\newcommand{\Z}{\mathbb{Z}}

\newcommand{\h}{\mathbb{H}}
\newcommand{\A}{\mathbb{A}}
\newcommand{\im}{\text{Im}}
\renewcommand{\O}{\mathcal{O}}

\newcommand{\calO}{\mathcal{O}}
\newcommand{\calZ}{\mathcal{Z}}
\newcommand{\fraka}{\mathfrak{a}}

\newcommand{\frakd}{\mathfrak{d}}
\newcommand{\frakn}{\mathfrak{n}}
\newcommand{\frakm}{\mathfrak{m}}
\newcommand{\frakp}{\mathfrak{p}}
\newcommand{\frakq}{\mathfrak{q}}
\newcommand{\Cl}{\mathrm{Cl}}
\newcommand{\bj}{\mathbf{j}}
\newcommand{\pr}{\operatorname{pr}}
\newcommand{\ord}{\mathrm{ord}}
\newcommand{\real}{\text{Re}}

\newcommand{\PSL}{\operatorname{PSL}}
\newcommand{\SL}{\operatorname{SL}}
\newcommand{\GL}{\operatorname{GL}}
\newcommand{\U}{\operatorname{U}}

\newcommand{\xqedhere}[2]{%
  \rlap{\hbox to#1{\hfil\llap{\ensuremath{#2}}}}}

\newtheorem{thm}{Theorem}

\newtheorem{lem}[equation]{Lemma}
\newtheorem{prop}[equation]{Proposition}

\newtheorem{defi}[equation]{Definition}

\numberwithin{equation}{section}

\title{Quantum ergodicity of Eisenstein series for Bianchi groups}
\author{Doyon Kim}
\address{Mathematisches Institut, Endenicher Allee 60, 53115 Bonn, Germany
}
\email{kimdoyon@math.uni-bonn.de}
\author{Youngmin Lee}
\address{Department of Mathematics, Kyonggi University,	154-42, Gwanggyosan-ro, Yeongtong-gu, Suwon-si, Gyeonggi-do, Republic of Korea}
\email{youngminlee@kyonggi.ac.kr}

\date{\today}

\begin{document}

\begin{abstract}
We prove the quantum ergodicity of Eisenstein series on the arithmetic hyperbolic 3-manifold $\operatorname{PSL}_2(\mathcal{O}_F)\backslash \mathbb{H}^3$, where $F$ is an imaginary quadratic field with ring of integers $\mathcal{O}_F$ and class number $h_F\geq 1$. This extends the work of Koyama, who proved the result in the case $h_F=1$, and establishes the first instance of quantum ergodicity of Eisenstein series over number fields with nontrivial class groups.
\end{abstract}

\maketitle

\section{Introduction}
Let $\h^2$ denote the upper half plane, and let $\Gamma$ be a discrete subgroup of $\PSL_2(\R)$ such that $X=\Gamma\backslash \h^2$ is compact. Let $\Delta$ be the Laplace operator on $X$, and let $(\phi_j)_{j=1}^\infty$ be an orthonormal basis of the set of square-integrable Laplace eigenfunctions of $X$, where each $\phi_j$ has eigenvalue $\lambda_j$. Assume that $(\lambda_j)_{j=1}^\infty$ is nondecreasing. Let $d \mu_j$ denote the probability measure, $d \mu_j=|\phi_j|^2 dV$. Shnirelman, Colin de Verdi\`ere, and Zelditch \cite{verdiere1985ergodicite,snirelman1974ergodic, zelditch1987uniform} proved that there exists a subsequence $(\phi_{j_k})_{k=1}^\infty$ of density one such that the associated probability measures converge weakly to the normalized volume measure, in particular,
\[\lim_{k\to \infty} \mu_{j_k}(A)=\frac{\mathrm{Vol}(A)}{\mathrm{Vol}(X)}\]
for any continuity set $A\subset X$. This phenomenon has been termed \emph{quantum ergodicity}. \par
In \cite{rudnick1994behaviour}, Rudnick and Sarnak considered the noncompact surfaces of the form $Y=\Gamma\backslash \h^2$, where $\Gamma$ is a congruence subgroup. They conjectured that 
\[\lim_{j\to \infty} \mu_{j}(A)=\frac{\mathrm{Vol}(A)}{\mathrm{Vol}(Y)}.\]
This strengthneing of quantum ergodicity is referred to as \emph{quantum unique ergodicity}. For the surface $Y=\PSL_2(\Z)\backslash \h^2$, the quantum unique ergodicity of Hecke--Maass cusp forms was proved in the works of \cite{lindenstrauss2006invariant} and \cite{soundararajan2010que}, and in \cite{holowinsky2010mass} an analogue of quantum unique ergodicity of holomorphic Hecke eigenforms was established. \par
In the noncompact setting, the surfaces possess a continuous spectrum as well as the discrete spectrum. In \cite{luosarnak1995qe}, Luo and Sarnak proved the continuous spectrum analogue of quantum ergodicity. Let $E(z,s)$ be the Eisenstein series on $\PSL_2(\Z)$, and let $d\mu_t=|E(z,1/2+it)|^2 dV$. They proved
\[\lim_{t\to\infty} \frac{\mu_t(A)}{\mu_t(B)}=\frac{\mathrm{Vol}(A)}{\mathrm{Vol}(B)}\]
for any compact Jordan measurable subsets $A,B\subset \PSL_2(\Z)\backslash \h^2$. \par
The quantum ergodicity of Eisenstein series admits natural generalizations to arithmetic quotients associated with number fields. In this direction, Koyama \cite{koyama2000qe} established quantum ergodicity for Eisenstein series attached to imaginary quadratic fields of class number one, covering a finite list of fields, and Truelsen \cite{truelsen2011qe} proved an analogous result for Eisenstein series on Hilbert modular groups over totally real fields with narrow class number one.
\par 
In the present work, we prove the quantum ergodicity of the Eisenstein series on $\PSL_2(\calO_F)\backslash \h^3$, where $\h^3$ is the hyperbolic 3-space and $F$ is an imaginary quadratic field with ring of integer $\calO_F$ with class number $h_F>1$. This provides the first extension of quantum ergodicity for Eisenstein series beyond the class number one setting.
\begin{thm}\label{thm:qe}
Let $F$ be an imaginary quadratic field with class number $h_F>1$, and let $X = \PSL_2(\calO_F)\backslash \mathbb \h^3$. For a cusp $\eta$ of $\PSL_2(\calO_F)$, let $E_{\eta}(v,s)$ denote the Eisenstein series associated to $\eta$. Then, for any compact Jordan measurable subsets $A,B\subset X$, we have
    \[\lim_{t\to\infty}\frac{\mu_{\eta,t}(A)}{\mu_{\eta,t}(B)}=\frac{\mathrm{Vol}(A)}{\mathrm{Vol}(B)},\]
    where $\mu_{\eta,t}=|E_{\eta}(v,1+it)|^2 dV$.
\end{thm}
The precise definition of $E_{\eta}(v,s)$ is given in Section~\ref{sec:arith}. Our proof works equally well for the $h_F=1$ case, but in this case the proof coincides with that of Koyama. We therefore restrict our attention to $h_F>1$.
\par
In quantum ergodicity problems, the proofs typically follow a common general strategy. First, we use the spectral decomposition to reduce the problem to the study of the asymptotic behavior of integrals of the form
\begin{equation}\label{basicform}\int_{X}f d\mu_{t},\end{equation}
with $X$ the underlying manifold and $t$ the limiting parameter, 
where $f$ is either a cuspidal Laplace eigenfunction (which we refer to as a Maass cusp form) or an incomplete Eisenstein series. After that, we unfold the integral and use the orthogonality of Fourier modes to express the integral in terms of $L$-functions. The desired asymptotics then follow from the subconvexity bounds for the associated $L$-functions. This approach was employed by Luo and Sarnak in \cite{luosarnak1995qe} for Eisenstein series on $\PSL_2(\Z)\backslash \h^2$, and subsequently adapted by Koyama \cite{koyama2000qe} to establish quantum ergodicity for imaginary quadratic fields of class number one. \par 
However, when the class number exceeds one, this process does not go as smoothly, since the classical framework for automorphic forms on $\Gamma\backslash \h^3$ fails to reflect the global Hecke structure. In this case, Hecke operators are indexed by ideals, and non-principal ideals cannot be realized as single double cosets. Consequently, the notion of a Hecke--Maass form becomes ambiguous at the classical level. Another difficulty arises from the fact that Fourier expansions at a fixed cusp encode only partial ideal-theoretic information, while the relevant $L$-functions are defined globally over all ideal classes.
\par 
These issues do not arise when the class number is one, because in this case every ideal is principal, classical Hecke operators admit a double-coset description, and a single cusp suffices to recover the full Hecke data. In particular, a classical Hecke–Maass cusp form on $\Gamma\backslash\h^3$ lifts naturally to a single adelic automorphic form, and the Fourier expansion at infinity captures the full global Hecke data. \par
It is essential to introduce the adelic framework for the higher class number case in order to restore the Hecke action. However, the construction of an adelic realization of a Maass cusp form is not merely a matter of careful bookkeeping. Although it is well known that an adelic automorphic form on $\GL_2(\A_F)$ corresponds to an $h$-tuple of classical functions satisfying suitable modularity properties (see, for instance, \cite{boreljacquet1979automorphic}), assigning such an $h$-tuple to a given Maass cusp form $f$ is not canonical. As we shall see later, it turns out that $f$ cannot be lifted to a single adelic automorphic form when the class number is even. \par 
Our main novelty lies in the remedy of this obstruction, by constructing an adelic realization of a classical Maass cusp form as a linear combination of adelic automorphic forms with varying central characters. Since odd Maass cusp forms do not contribute to the integral \eqref{basicform}, we restrict our attention to even Maass cusp forms. 
\begin{thm}\label{thm:f_adelic}
Let $f$ be an even Maass cusp form on $\PSL_2(\calO_F)\backslash\h^3$. Then $f(v)=f(z+r\bj)$ can be expressed as a finite linear combination of functions of the form
\[\Phi\biggl(i_\infty \begin{pmatrix}
    r & z \\ 0 & 1
\end{pmatrix}\biggr),\]
where $\Phi$ is an adelic Hecke eigenform whose central character is a Hecke character lifted from a class group character. Here, $i_\infty (\cdot)$ denotes the adelic embedding with trivial finite component.
\end{thm}
A more precise version of this statement, which treats the functions $f(A_j^{-1} v)$ associated to different cusps, is given in Proposition~\ref{prop:f_adelic}. This restores a natural Hecke
action at the adelic level. After this, we introduce a cusp-averaging technique that allows global $L$-function data to be recovered from cuspwise Fourier expansions.
\par
In Section~\ref{sec:arith}, we provide basic notation and background on Eisenstein series on $\Gamma\backslash\h^3$. In Section~\ref{sec:Maass}, we construct the adelic realization of Maass cusp forms on $\Gamma\backslash\h^3$. We record bounds of $\GL_1$ and $\GL_2$ $L$-functions in Section~\ref{sec:lfunctions}, and prove the quantum ergodicity of Eisenstein series in Section~\ref{sec:qe} using these bounds. 
\section*{Acknowledgements}
The first author was supported by ERC Advanced Grant 101054336 and Germany’s 
Excellence Strategy grant EXC-2047/1 - 390685813. The second author was supported by a KIAS Individual Grant (MG086302) at Korea Institute for Advanced Study. We thank Valentin Blomer and Edgar Assing for helpful comments.

\section{Arithmetic hyperbolic 3-manifolds}\label{sec:arith}
We regard the hyperbolic 3-space $\h^3$ as a subset of Hamiltonian quaternions, and write an element $v$ of $\h^3$ as $v=z+r\bj=x+yi+r\bj$ with $z\in \C$ and $r>0$. Accordingly, we use the notation 
\begin{equation}
    \real(z)=x, \quad \im(z)=y, \quad \Im(v)=r.
\end{equation}
The groups $\GL_2(\C)$ and $\PSL_2(\C)$ act on $\h^3$ by
\[gv=q^{-1} (av+d)(cv+d)^{-1} q, \quad g=\begin{pmatrix}
    a & b \\ c & d
\end{pmatrix},\]
where $q=\sqrt{\det g}$ and the inverse is taken in the skew field of quaternions. The volume element of $\h^3$ is given by $dV=dx dy dr/r^3$ and the corresponding Laplace operator is
\[\Delta=r^2\left(\frac{\partial^2}{\partial x^2}+\frac{\partial^2}{\partial y^2}+\frac{\partial^2}{\partial r^2}\right)-r \frac{\partial}{\partial r}.\]
Let $F$ be an imaginary quadratic field and $\calO_F$ be the ring of integers of $F$. Let $\Gamma=\PSL_2(\O_F)$. Then $\Gamma$ is a discrete subgroup of $\PSL_2(\C)$. We write $\mathrm{Cl}_{F}$ for the class group of $F$, $\mathcal{M}$ be the set of fractional ideals, $d_F$ for the discriminant of $F$, and $\frakd$ for the different ideal of $F$. For $\frakm\in \mathcal{M}$, we let $N(\frakm)$ denote the norm of $\frakm$.
\par 
Let $h=h_F$ be the class number of $F$. Throughout this paper, we assume $h>1$. Then it is well known that the number of $\Gamma$-inequivalent cusps is $h$, because there is a bijection between the class group $\Cl_F$ and the set of $\Gamma$-equivalence classes of the cusps of $\Gamma$. We say that a matrix $\begin{pmatrix}
    a & b \\ c & d
\end{pmatrix}$ is quasi-integral if 
\begin{equation}\label{quasiintegral}ac, ad, bc, bd\in \calO_F.\end{equation}
For a cusp $\eta$ of $\Gamma$, we attach a quasi-integral matrix
\begin{equation}\label{Aetadef}
 A_\eta=\begin{pmatrix}
    a & b \\ 1 & -\eta
\end{pmatrix}\in \SL_2(F),  
\end{equation}
which is possible by  \cite[Lemma~2.9, p.~368]{EGM2},
and a fractional ideal
\begin{equation}\label{frakmeta}
\frakm_\eta=\langle 1,\eta \rangle.   
\end{equation}
Check that $\eta=A_\eta^{-1}\infty$. If $\eta=\infty$, we let $A_\eta=I_2$ and $\frakm_\eta=\calO_F$. We write $\Gamma_\eta$ for the stabilizer of $\eta$, so 
\begin{equation}\label{Gammaeta}\Gamma_\eta=\{\gamma\in\Gamma:\gamma \eta=\eta\}.\end{equation}
By \cite[Lemma~2.2, p.~365]{EGM2}, we have
\begin{equation}\label{Gammaetaconj}A_\eta\Gamma_{\eta} A_\eta^{-1}=\left\{\begin{pmatrix}
    1 & b \\ 0 & 1
\end{pmatrix}: b\in\frakm_\eta^{-2}\right\}.\end{equation}
For a cusp $\eta$ of $\Gamma$, the Eisenstein series associated to the cusp $\eta$ is defined as follows.
\begin{defi}
For a cusp $\eta$ of $\Gamma$, the Eisenstein series associated to the cusp $\eta$ is given by
 \begin{equation} \label{Eetadef}
   E_{\eta}(v,s)=\sum_{\gamma\in \Gamma_{\eta}\backslash \Gamma} \Im(A_\eta \gamma v)^{s},
 \end{equation}
where $A_\eta$ is the quasi-integral matrix defined in \eqref{Aetadef}.
\end{defi}
By definition~(2.5) (p.~99), Theorem~1.7 (p.~363), and Theorem~3.8 (p.~377) of \cite{EGM2}, the Eisenstein series $E_\eta(v,s)$ has a meromorphic continuation to the whole $s$-plane, holomorphic for $\real(s)>1$ except for a simple pole at $s=2$. At $s=2$, we have
\[\mathrm{Res}_{s=2}E_\eta (v,s)=\frac{2\pi^2 N(\frakm_\eta^{-2})}{|d_F|\zeta_F(2)},
\]
where $\zeta_F(s)$ denotes the Dedekind zeta function of $F$. Therefore, by \cite[Proposition~2.2 (p.~245), Theorem~3.4 (p.~267)]{EGM2}, the spectrum of $\Delta$ consists of three types: the constant function, the cuspidal Laplace eigenfunctions, which we call Maass cusp forms, and Eisenstein series attached to each cusp. In view of this decomposition, to establish quantum ergodicity it suffices to understand the integrals of the form \eqref{basicform} where $f$ is either a Maass cusp form or an incomplete Eisenstein series attached to a cusp.
\section{Maass cusp forms on $\Gamma\backslash \h^3$} \label{sec:Maass}
When $h>1$, the Hecke operators cannot be defined in a classical way that uses double cosets, since the Hecke operators are indexed by the integral ideals, whereas only the principal ideals admit such classical definition. In order to relate the Fourier coefficients of Maass cusp forms at each cusp with Hecke eigenvalues, we need to adopt the adelic viewpoint. The main goal of this section is to prove Proposition~\ref{prop:f_adelic}, which describes the Maass cusp forms viewed at each cusp in terms of adelic Hecke eigenforms.
\par 
Let $\A=\A_F$ be the ring of adeles over $F$ with archimedean place $F_\infty=\C$ and finite places $F_\frakp$ indexed by prime ideals of $F$. Let $\calO_\frakp$ be the local ring of integers of $F_\frakp$ and let $\varpi_\frakp$ denote the local uniformizer of $\calO_{\frakp}$. Let $\A_{\textrm{fin}}=\prod'_{\nu\nmid \infty} F_\nu$ denote the ring of finite adeles. For a finite idele $x\in \A_{\textrm{fin}}^\times$, we let $\fraka(x)$ denote the fractional ideal of $F$ associated to $x$. We write $G(\A)=\GL_2(\A_F)$, $G(F)=\GL_2(F)$, $G_\infty=\GL_2(\C)$, $G_\nu=\GL_2(F_\nu)$, $K_\infty=\U_2(\C)$, and $K_f=\prod_{\nu\nmid \infty}\GL_2(\calO_\nu)$. We use $\calZ_\infty\cong \C^\times$ and $\calZ_\A\cong \A_F^\times$ to denote the centers of $G_\infty$ and $G(\A)$, respectively. For every place $\nu$ of $F$, we let $\pr_\nu$ be the projection of 
$G(\A)$ onto $G_\nu$ defined by
\[
\pr_\nu(g)=g_\nu,\quad g=(g_\nu)_\nu\in G(\A).
\]
We regard $G(F)$ as a subgroup of $G(\A)$ via the diagonal embedding. Also, we let $i_
\infty:G_\infty\to G(\A)$ denote the archimedean embedding that maps $g\in G_\infty$ to the adelic element whose archimedean component is 
$g$ and whose non-archimedean components are trivial. 
\subsection{Correspondence between adelic and classical functions} It is well-known that an adelic automorphic form on $G(\A)$ corresponds to an $h$-tuple of functions on $G_\infty$ \cite{boreljacquet1979automorphic}. However, the reverse construction presents a difficulty. While strong approximation enables us to project an adelic automorphic form to an $h$-tuple of functions on $G_\infty$, starting from a single Maass cusp form $f$ and producing an $h$-tuple with the correct transformation properties is not immediate. We adopt the setting of \cite{bygott1998phd} to develop an explicit adelic realization of $f$. Let $\mathcal{C}=\{\frakm_1,\ldots,\frakm_h\}$ be a set of class group representatives, so that $\Cl_F=\{[\frakm_1],\ldots,[\frakm_h]\}$. Let $\theta_j$ be the finite idele associated to $\frakm_j$ defined by $\theta_j=\prod_{\frakp}\varpi_\frakp ^{n_{\frakp,j}}$, where $\frakm_j=\prod_{\frakp} \frakp ^{n_{\frakp,j}}$ with $n_{\frakp,j}=0$ for all but finitely many $\frakp$. Let $t_j=\begin{pmatrix}
    \theta_j^{-1} & 0 \\ 0 & 1
\end{pmatrix}\in \GL_2(\A_{\textrm{fin}})$. We can write $G(\A)$ as a union of $h$ disjoint double cosets
\[G(\A)=\bigsqcup_{1\leq j\leq h} G(F) t_j (G_\infty K_f),\]
by the strong approximation theorem. Let $\Gamma^{[j]}$ and $\widetilde{\Gamma}^{[j]}$ be subsets of $G_\infty$ defined by
\begin{equation}\label{eq:Gammaj}
\Gamma^{[j]}=G(F) \cap t_j(G_\infty K_f) t_j^{-1},\quad \widetilde{\Gamma}^{[j]}=G(F)\cap t_j(G_\infty K_f)\mathcal{Z}_\A t_j^{-1},\end{equation}
where the intersection is taken adelically after diagonal embedding of $G(F)$, and the resulting subgroup is then viewed inside $G_\infty$ via the archimedean projection. In \cite{bygott1998phd}, Bygott establishes the bijection between the set of adelic automorphic functions on $G(\A)$ with trivial central character and the set of certain $h$-tuples $(\phi_1,\ldots,\phi_h)$ of functions on $G_\infty$, where each $\phi_j$ is left $\widetilde{\Gamma}^{[j]}$-invariant. To describe the bijection, we first recall the notion of $\calZ$-compatibility. Let $S$ be the set of indices such that $[\frakm_j]$ with $j\in S$ represents the ideal classes modulo squares. For any $j\notin S$, there exists a unique $i\in S$ for which there exists $z\in\calZ_\A$ with \[z t_j\in G(F) t_i (G_\infty K_f).\]
This index $i$ is characterized by the relation $[\frakm_j^{-1}\frakm_i]=[\frakm_k]^2$ for some $1\leq k\leq h$, and a possible choice of $z$ is $\theta_k I_2$ \cite[Proposition 88]{bygott1998phd}. 
\begin{defi}\label{def:Zcompatible}
    We say that $(\phi_1,\ldots,\phi_h)$ is $\calZ$-compatible if for every $j\notin S$, the unique index $i\in S$ for which $[\frakm_j^{-1}\frakm_i]\in \Cl_F$ is a square satisfies
    \begin{equation}\label{eq:Zcompatible}\phi_j (x)=\phi_i(w_\infty x),\end{equation}
    where $w=w_\infty w_f\in G_\infty K_f$ is any matrix such that \[z t_j=\delta t_i w \quad \text{for some} \quad z\in \calZ_\A, \quad \delta\in G(F).\]
\end{defi}
In particular, $\phi_j$ is uniquely determined by $\phi_i$ by the relation \eqref{eq:Zcompatible}. Also, it is shown in \cite[Proposition 88]{bygott1998phd} that the relation does not depend on the choice of $w$. Below is a straightforward application of \cite[Theorem 92]{bygott1998phd}.
\begin{lem}\label{lem:bijection}
    The map $\Phi\mapsto (\phi_1,\ldots,\phi_h)$ given by
    \[\phi_j(g_\infty)=\Phi(g_\infty t_j)\] 
    and its inverse \[\Phi(\gamma t_j g_\infty k_f)=\phi_j(g_\infty), \quad \gamma\in G(F), \quad g_\infty \in G_\infty, \quad k_f\in K_f\]
    gives a bijection between the set of left $G(F)$-invariant, right $K_f$-invariant, and $\calZ_\A$-invariant functions $\Phi$ on $G(\A)$ and the set of $\calZ$-compatible $h$-tuples $(\phi_1,\ldots,\phi_h)$ of functions on $G_\infty$ where each $\phi_j$ is left-invariant under $\widetilde{\Gamma}^{[j]}$ and invariant under $\calZ_\infty$.
\end{lem}
Our goal is to construct each of the $h$ components from a Maass form $f$ via the matrices \eqref{Aetadef}. However, there is an obstruction to this construction. While the bijection of Lemma~\ref{lem:bijection} applies to functions that are left $\calZ_\infty \widetilde{\Gamma}^{[j]}$-invariant, the classical functions obtained from $f$ by the slash operators corresponding to $A_\eta$ are, in general, only left-invariant under the smaller group 
$\calZ_\infty \Gamma^{[j]}$. As a consequence, these components do not necessarily correspond to an adelic automorphic form with trivial central character. To resolve this, we shall generalize the bijection described in Lemma~\ref{lem:bijection} to allow adelic automorphic forms with nontrivial central character. \par
For this, let us first describe the relation between the two groups $\calZ_\infty \widetilde{\Gamma}^{[j]}$ and
$\calZ_\infty \Gamma^{[j]}$. Let $\Cl_F[2]$ denote the 2-torsion subgroup of $\Cl_F$.
\begin{lem}\label{lem:gammasqnew}
If $g\in \mathcal{Z}_\infty \widetilde{\Gamma}^{[j]}$ then $g^2\in \mathcal{Z}_\infty \Gamma^{[j]}$.
\end{lem}
\begin{proof}
Write $g=\alpha \begin{pmatrix}
    a & b \\ c & d
\end{pmatrix}$ with $\alpha \in \C^\times $ and $\begin{pmatrix}
    a & b \\ c & d
\end{pmatrix}\in \widetilde{\Gamma}^{[j]}$. Then
\[t_j^{-1} \begin{pmatrix}
    a & b \\ c & d
\end{pmatrix}t_j= \begin{pmatrix}
    a& b \theta_j \\ c\theta_j^{-1} & d
\end{pmatrix}\in G_\infty K_f \mathcal{Z}_\A,\]
so there exists $z=z_\infty z_f \in \A^\times $ such that
\[ \pr_\frakp   \begin{pmatrix}
  a z_f^{-1}& b \theta_j z_f^{-1} \\ c\theta_j^{-1} z_f^{-1} & d z_f^{-1}
\end{pmatrix} \in \GL_2(\O_\frakp)\]
for every $\frakp$. Let $\ord_\frakp (z_f)=m_\frakp$. We have
\[2\min \left(\ord_\frakp (a),\ord_\frakp (b)+n_{\frakp,j},\ord_\frakp (c)-n_{\frakp,j},\ord_\frakp (d)\right)=\ord_\frakp (ad-bc)=2m_\frakp.\]
Also, since every entry of the matrix
\[\pr_\frakp \Bigl( t_j^{-1} \begin{pmatrix}
    a & b \\ c & d
\end{pmatrix}^2t_j\Bigr)=\begin{pmatrix}
    a^2+bc& b(a+d) \varpi_\frakp^{n_{\frakp,j}} \\ c(a+d)\varpi_\frakp^{-n_{\frakp,j}} & bc+d^2
\end{pmatrix} \]
lies in the ideal $(a, b\varpi_\frakp^{n_{\frakp,j}},c\varpi_\frakp^{-n_{\frakp,j}},d)^2\calO_\frakp=(ad-bc)\calO_\frakp$, we can write
\[t_j^{-1} \begin{pmatrix}
    a & b \\ c & d
\end{pmatrix}^2t_j= \lambda  \begin{pmatrix}
    a'& b' \\ c' & d'
\end{pmatrix},\]
where $\lambda=ad-bc\in F^\times$, and $\begin{pmatrix} a' & b'\\
c' & d'\end{pmatrix}\in \GL_2(\O_\frakp)$ for every $\frakp$. Thus 
\[\lambda^{-1} \begin{pmatrix}
    a & b \\ c & d
\end{pmatrix}^2 \in t_j G_\infty K_f t_j^{-1} \cap G(F)=\Gamma^{[j]}.\]
We conclude that 
\[g^2=(\alpha^2 \lambda) \lambda^{-1} \begin{pmatrix}
    a & b \\ c & d
\end{pmatrix}^2 \in \mathcal{Z}_\infty \Gamma^{[j]}.\]
This finishes the proof.
\end{proof}

\begin{lem}\label{lem:Gammajcosetnew}
For any $1\leq j\leq h$, the map
\begin{equation}\label{rhoj}
\rho_j : \mathcal{Z}_\infty \widetilde{\Gamma}^{[j]} \to \Cl_F[2],
\quad
\alpha g= \alpha  t_j g_\infty k_f z t_j^{-1} \mapsto [\fraka(z_f)],
\end{equation}
where $\alpha\in\calZ_\infty$, $g\in G(F)$, $g_\infty\in G_\infty$, $k_f\in K_f$, $z=z_\infty z_f\in \calZ_\A$,
is a well-defined group homomorphism, and induces an isomorphism
\[
\calZ_\infty \widetilde{\Gamma}^{[j]} / \calZ_\infty \Gamma^{[j]}
\cong \Cl_F[2].
\]
\end{lem}
\begin{proof}
For $\alpha g=\alpha  t_j g_\infty k_f z t_j^{-1}\in \mathcal{Z}_\infty \widetilde{\Gamma}^{[j]}$, we have 
\[[\fraka(z_f)^2]=[\langle \det g\rangle],\]
hence $[\fraka(z_f)]\in\Cl_F[2]$. It is straightforward to check that $\rho_j$ is a well-defined group homomorphism.
\par 
To prove the isomorphism, we shall show that $\ker(\rho_j)=\calZ_\infty  \Gamma^{[j]}$. Suppose that $\alpha g \in\calZ_\infty \Gamma^{[j]}$. Then $\alpha g=\beta g'$ for some $\beta\in \calZ_\infty$ and $g'=t_j g'_\infty k'_f t_j^{-1}\in \Gamma^{[j]}$. Let $\lambda=\alpha^{-1} \beta$. Since $g^{-1}g'=\lambda^{-1} I_2$, we have $\lambda\in F^\times$. Also, since
\[\begin{aligned}
    g^{-1} g'= z^{-1}t_j ( g_\infty^{-1}g'_\infty k_f^{-1}k'_f) t_j^{-1}\in \widetilde{\Gamma}^{[j]},
\end{aligned}\]
we deduce $\fraka(z_f)=\langle \lambda \rangle$. Conversely, if $\alpha g=\alpha t_j g_\infty k_f z t_j^{-1} \in \calZ_\infty \widetilde{\Gamma}^{[j]}$ satisfies $\fraka( z_f) =\langle \lambda \rangle$ for some $\lambda \in F^\times$, then we have $\lambda^{-1} g\in \Gamma^{[j]}$, thus $\alpha g=\alpha \lambda (\lambda^{-1} g)\in \calZ_\infty \Gamma^{[j]}$. We conclude that $\ker(\rho_j)=\calZ_\infty \Gamma^{[j]}$. \par
It remains to prove that $\im(\rho_j)=\Cl_F[2]$. Since we can choose a prime ideal representative for any ideal class, it suffices to show that for any prime ideal $\frakp\in \Cl_{F}[2]$ we can find $\begin{pmatrix}
    a & b \\ c & d
\end{pmatrix}\in G(F)$ such that
\[t_j^{-1} \begin{pmatrix}
    a & b \\ c & d
\end{pmatrix}t_j =\begin{pmatrix}
    a& b \theta_j \\ c\theta_j^{-1} & d
\end{pmatrix}=g_\infty k_f z\]
for some $g_\infty\in G_\infty$, $k_f\in K_f$, $z\in \calZ_\A$ with $[\fraka( z_f)]= [\frakp]$. Choose a prime ideal $\frakn_2$ in $[\frakp \frakm_j]$ with $\frakn_2\neq \frakp$. Then $\frakn_2\frakp \frakm_j^{-1}=(\lambda_2)$ for some $\lambda_2\in F$. Next, choose a prime ideal $\frakn_1$ such that $[\frakn_1]=[\frakp]$ and $\frakn_1\notin \{\frakn_2,\frakp\}$. By the Chinese Remainder Theorem, we can find $\mu\in \calO_F$ such that $\mu\equiv 1 \pmod {\frakn_2}$, $\mu\equiv 0 \pmod {\frakp \frakn_1}$. Let $\frakn_3=(\mu-1)\frakn_2^{-1}$, $\frakn_4=(\mu)\frakn_1^{-1}$. We have 
\[\frakp^2=(\lambda), \quad \frakp\frakn_1=(\lambda_1),\quad \frakp\frakn_4=(\lambda_4)\]
for some $\lambda,\lambda_1,\lambda_4\in \calO_F$. Multiplying by units if necessary, we may assume that $\lambda_1$ and $\lambda_4$ satisfy $\lambda_1\lambda_4=\lambda\mu$. Let
\[\lambda_3=\frac{\lambda_1\lambda_4-\lambda}{\lambda_2}.\]
We have
\[(\lambda_3)=(\lambda)(\mu-1)\frakp^{-1}\frakn_2^{-1}\frakm_j=\frakp \frakn_3\frakm_j.\]
Let $g=\begin{pmatrix}
    a & b \\ c & d
\end{pmatrix}=\begin{pmatrix}
    \lambda_1 & \lambda_2 \\ \lambda_3 & \lambda_4
\end{pmatrix}$. Since
$\det(g)=\lambda$, we have $\ord_\frakp (\det g)=2$ and $\ord_\frakq(\det g)=0$ for every prime $\frakq\neq \frakp$. Furthermore, since each $\frakn_j$ is integral and $\frakn_1$, $\frakn_2$ are distinct prime ideals, we have
\[\begin{aligned}&\min \bigl(\ord_\frakp (a),\ord_\frakp (b\theta_j),\ord_\frakp (c\theta_j^{-1}),\ord_\frakp (d)\bigr) \\ 
&=\min\left(\ord_\frakp (\frakp\frakn_1),\ord_\frakp (\frakp\frakn_2),\ord_\frakp (\frakp\frakn_3),\ord_\frakp (\frakp\frakn_4)\right)=1, \end{aligned}\]
and
\[\begin{aligned}&\min \bigl(\ord_\frakq (a),\ord_\frakq (b\theta_j),\ord_\frakq  (c\theta_j^{-1}),\ord_\frakq  (d)\bigr) =0 \end{aligned}\]
for every prime $\frakq\neq \frakp$. We conclude that $[\fraka( z_f)]= [\frakp]$.
\end{proof}
From the above two lemmas, we obtain the following result.
\begin{lem}\label{lem:Gammatildejeigen}
    Let $h'=\Cl_F[2]$ be the number of $2$-torsion elements of $\Cl_F$, and let $\chi_1,\ldots,\chi_{h'}$ be characters on $\Cl_F[2]$. Any $\calZ_\infty \Gamma^{[j]}$ invariant function $\phi$ can be written as a linear combination of functions $\phi_m$, $1\leq m\leq h'$ that transform under $\calZ_\infty \widetilde{\Gamma}^{[j]}$ by 
\begin{equation}\label{phiitransform}\phi_m(\gamma g)=\chi_m\left(\rho_j(\gamma)\right)\phi_m(g), \quad \gamma\in \calZ_\infty \widetilde{\Gamma}^{[j]},\end{equation}
    where $\rho_j$ is the homomorphism defined in \eqref{rhoj}.
\end{lem}
\begin{proof}
By Lemma~\ref{lem:Gammajcosetnew}, there exists a set $\{\gamma_1,\dots, \gamma_{h'}\}$ of representatives for $\mathcal{Z}_{\infty}\widetilde{\Gamma}^{[j]}/\mathcal{Z}_{\infty}\Gamma^{[j]}$. For each $1\leq m\leq h'$, define
\[ \phi_m(g)=\frac{1}{h'}\sum_{i=1}^{h'}\chi^{-1}_m\left(\rho_j(\gamma_i)\right) \phi(\gamma_i g). \]
It is straightforward to check that $\phi=\sum_{m=1}^{h'} \phi_m$. To verify \eqref{phiitransform}, fix $\gamma\in \mathcal{Z}_{\infty}\widetilde{\Gamma}^{[j]}$. Then there is a permutation $\sigma$ on $\{1,\dots,h'\}$ such that $\gamma_i \gamma \in \gamma_{\sigma(i)}\bigl(\mathcal{Z}_{\infty}\Gamma^{[j]}\bigr)$.
Also, since $\mathcal{Z}_{\infty}\Gamma^{[j]}$ is a normal subgroup of $\mathcal{Z}_{\infty}\widetilde{\Gamma}^{[j]}$, we have $\gamma_{\sigma(i)}\big(\mathcal{Z}_{\infty}\Gamma^{[j]}\bigr)= \bigl(\mathcal{Z}_{\infty}\Gamma^{[j]}\bigr)\gamma_{\sigma(i)}$. It follows that
\begin{equation*}
    \begin{aligned}
        \phi_m(\gamma g)
        &=\frac{1}{h'}\chi_m\bigl(\rho_j(\gamma)\bigr)\sum_{i=1}^{h'}\chi^{-1}_m\bigl(\rho_j(\gamma_i\gamma)\bigr) \phi(\gamma_i \gamma g)\\
        &=\frac{1}{h'}\chi_m\bigl(\rho_j(\gamma)\bigr)\sum_{i=1}^{h'}\chi_m^{-1}(\rho_j\bigl(\gamma_{\sigma(i)})\bigr) \phi(\gamma_{\sigma(i)} g)\\
        &= \chi_m(\rho_j(\gamma)) \phi_m(g),
    \end{aligned}
\end{equation*}
hence $\phi_m$ satisfies \eqref{phiitransform}. This finishes the proof.
\end{proof}

Using the group homomorphism $\rho_j$, we extend the notion of $\calZ$-compatibility to functions with nontrivial central characters. Recall from above Definition~\ref{def:Zcompatible} that $S$ is the set of indices such that $[\frakm_j]$ with $j\in S$ represents the ideal classes modulo squares.
\begin{defi}\label{def:Zchicompatible}
Let $\chi$ be a class group character. We say that $(\phi_1,\ldots,\phi_h)$ is $(\calZ,\chi)$-compatible if for every $j\notin S$, the unique index $i\in S$ for which $[\frakm_j^{-1}\frakm_i]\in \Cl_F$ is a square satisfies
\begin{equation}\label{eq:Zchicompatible}\phi_j (x)=\chi^{-1}\bigl(\fraka(z_f)\bigr)\phi_i(w_\infty x),\end{equation}
    where $w=w_\infty w_f\in G_\infty K_f$ is any matrix such that 
    \[z t_j=\delta t_i w \quad \text{for some} \quad z=z_\infty z_f\in \calZ_\A, \quad \delta\in G(F).\]
\end{defi}
Again, $\phi_j$ is uniquely determined by $\phi_i$ by the relation \eqref{eq:Zchicompatible}, and $\phi_j$ does not depend on the choice of $w$. With a little modification to \cite[Theorem 92]{bygott1998phd}, we obtain the following bijection.
\begin{lem}\label{lem:chibijection}
   Let $\chi$ be a class group character and let $\omega_\chi$ be the Hecke character lifted from $\chi$ with trivial archimedean part. The map $\Phi\mapsto (\phi_1,\ldots,\phi_h)$ given by
    \[\phi_j(g_\infty)=\Phi(g_\infty t_j)\] 
    and its inverse \[\Phi(\gamma t_j g_\infty k_f)=\phi_j(g_\infty), \quad \gamma\in G(F), \quad g_\infty \in G_\infty, \quad k_f\in K_f\]
    gives a bijection between the set of functions $\Phi$ on $G(\A)$ which are left $G(F)$-invariant, right $K_f$-invariant, and transforms under $\calZ_\A$ by the character $\omega_\chi$, and the set of $(\calZ,\chi)$-compatible $h$-tuples $(\phi_1,\ldots,\phi_h)$ of functions on $G_\infty$ for which each $\phi_j$ satisfies
    \[\phi_j(\gamma g)=\chi^{-1}\left(\rho_j(\gamma)\right)\phi_j(g)\]
    for every $\gamma\in \calZ_\infty \widetilde{\Gamma}^{[j]}$ and $g\in G_\infty$, where $\rho_j$ is the group homomorphism defined in \eqref{rhoj}.
\end{lem}

\subsection{Adelization of Maass forms on $\Gamma\backslash \h^3$}
Let $f$ be a Maass cusp form on $\Gamma\backslash \h^3$. Since $\h^3=G_\infty/(\calZ_\infty K_\infty)$, we can lift $f$ to a function on $G_\infty$ that is right $K_\infty$-invariant and invariant under the center $\calZ_\infty$, using the Iwasawa decomposition. By a slight abuse of notation, we denote this lift again by $f$. Observe that if $f$ is odd, then we have
\[\int_{\Gamma\backslash \h^3} f d\mu_{\eta,t}\equiv 0.\]
Accordingly, we shall assume that $f$ is even. Then we have
\[f\biggl(\begin{pmatrix}
    -1 & 0 \\ 0 & 1
\end{pmatrix}\begin{pmatrix}
    r & z \\ 0 & 1
\end{pmatrix}\biggr)=f\biggl(\begin{pmatrix}
    -1 & 0 \\ 0 & 1
\end{pmatrix}\begin{pmatrix}
    r & z \\ 0 & 1
\end{pmatrix}\begin{pmatrix}
    -1 & 0 \\ 0 & 1
\end{pmatrix}\biggr)=f\biggl(\begin{pmatrix}
    r & -z \\ 0 & 1
\end{pmatrix}\biggr)=f\biggl(\begin{pmatrix}
    r & z \\ 0 & 1
\end{pmatrix}\biggr)\]
for any $\begin{pmatrix}
    r & z \\ 0 & 1
\end{pmatrix}\in G_\infty$ with $r>0$, thus $f$ is left invariant under $\begin{pmatrix}
    -1 & 0 \\ 0 & 1
\end{pmatrix}$. Together with the left $\SL_2(\O_F)$-invariance, and the fact that the only units of $F$ are $\pm 1$ when $h>1$, we deduce that $f$ is left $\GL_2(\O_F)$-invariant. \par
Recall that $\mathcal{C}=\{\frakm_1,\ldots,\frakm_h\}$ is a set of class group representatives. For each $\frakm_j$, we associate a quasi-integral matrix
\begin{equation}\label{Ajdef}
    A_j=\begin{pmatrix}
        a_j & b_j \\ c_j & d_j
    \end{pmatrix}\in \SL_2(F)
\end{equation}
where the bottom row is chosen so that $\frakm_j=\langle c_j, d_j\rangle$, again using \cite[Lemma~2.9, p.~368]{EGM2} as in \eqref{Aetadef}. Note that if $\frakm_j=\frakm_\eta=\langle 1,\eta\rangle $ for a cusp $\eta$ then we can choose $A_j=A_\eta$ defined in \eqref{Aetadef}. We may assume that $\frakm_1=\calO_F$ and $A_1=I_2$. For $1\leq j\leq h$, we define \[f_j(g_\infty)=f(A_j^{-1} g_\infty).\]
Observe that $f_j$ is left-invariant under the group $A_j \GL_2(\O_F) A_j^{-1}$.
\begin{lem}\label{lem:AjGammaj}
For each $1\leq j\leq h$, we have
    \[A_j \GL_2(\O_F) A_j^{-1}=G(F) \cap t_j^2(G_\infty K_f) t_j^{-2}.\]
\end{lem}
\begin{proof}
Let us denote the set on the right hand side by $\widehat{\Gamma}_j$. First, let $\left(
\begin{smallmatrix}
 x & y \\
 z & w \\
\end{smallmatrix}
\right)\in \GL_2(\O_F)$ and $M=A_j \left(
\begin{smallmatrix}
 x & y \\
 z & w \\
\end{smallmatrix}
\right)A_j^{-1}$. To show that $A_j \GL_2(\O_F) A_j^{-1}\subseteq\Gamma_j$, it suffices to show that
    \[t_j^{-2} M t_j^2 =\begin{pmatrix}
 -y a_j c_j+x a_j d_j-w b_j c_j+z b_j d_j & \theta_j^2 (w a_j b_j-x a_j b_j+y a_j^2-z b_j^2) \\
 \theta_j^{-2}(-w c_j d_j+x c_j d_j-y c_j^2+z d_j^2) & y a_j c_j+w a_j d_j-x b_j c_j-z b_j d_j \\
\end{pmatrix}
\]
is in $ G_\infty K_f$. It follows from \eqref{quasiintegral} that the $(1,1)$-entry and $(2,2)$-entry of the matrix above is in $\O_F$. To check that the $(1,2)$-entry and $(2,1)$-entry are in $\O_\frakp$ for every finite place $\frakp$, observe that
\begin{equation}\label{npj}\min(\ord_\frakp(c_j), \ord_\frakp(d_j))=\ord_\frakp(\theta_j)=n_{\frakp,j},\end{equation}
since $\langle c_j,d_j\rangle=\frakm_j$. If the minimum is attained at $c_j$ then we have
\[\ord_{\frakp} \Bigl((t_j^{-2} M t_j^2 )_{1,2}\Bigr)=\ord_{\frakp}\Bigl((w-x)\frac{a_jc_j b_jc_j}{c_j^2}\theta_j^2+y\frac{a_jc_j a_jc_j}{c_j^2}\theta_j^2-z\frac{b_jc_j b_jc_j}{c_j^2}\theta_j^2\Bigr)\geq 0,\]
and if the minimum is attained at $d_j$ then we can replace $c_j$ in the above inequality with $d_j$ and arrive to the same conclusion.
Similarly, it follows directly from \eqref{npj} that $(t_j^{-2} M t_j^2 )_{2,1}\in \O_\frakp$. This proves the first direction of the inclusion. \par
The other direction of the inclusion is proved similarly. Let $X=\left(
\begin{smallmatrix}
 x & y \\
 z & w \\
\end{smallmatrix}
\right)\in \widehat{\Gamma}_j$. Then
\[t_j^{-2} \begin{pmatrix}
 x & y \\
 z & w \\
\end{pmatrix}t_j^2=\begin{pmatrix}
 x & y \theta_j^2 \\
 z\theta_j^{-2} & w \\
\end{pmatrix}\in G_\infty K_f,\]
hence for every finite place $\frakp$ we have $\ord_\frakp (x),\ord_\frakp(w)\geq 0$, and
\[\ord_\frakp (y)\geq -2\ord_{\frakp} (\theta_j)=-2n_{\frakp,j}, \quad \ord_\frakp (z)\geq 2\ord_{\frakp} (\theta_j)=2n_{\frakp,j}. \]
It follows from \eqref{quasiintegral}, \eqref{npj}, and the above inequalities that
\[A_j^{-1} X A_j=\begin{pmatrix}
 -z a_j b_j+x a_j d_j-w b_j c_j+y c_j d_j & -w b_j d_j+x b_j d_j-z b_j^2+y d_j^2 \\
 w a_j c_j-x a_j c_j+z a_j^2-y c_j^2 & z a_j b_j+w a_j d_j-x b_j c_j-y c_j d_j \\
\end{pmatrix}
\]
is in $\GL_2(\calO_\frakp)$ for every finite place $\frakp$. We conclude that $A_j^{-1} X A_j \in \GL_2(\O_F)$. This proves the other direction of the inclusion.
\end{proof}
Lemma~\ref{lem:AjGammaj} is somewhat unexpected, in the sense that although 
$A_j$ is the matrix which corresponds to $\frakm_j$, the function $f_j$ does not correspond to the $j$-th component of the associated $h$-tuple. Instead, it corresponds to the component indexed by the ideal class $[\frakm_j^2]$.

It follows from Lemma~\ref{lem:Gammajcosetnew} that $\mathcal{Z}_\infty \Gamma^{[j]}$ and $\mathcal{Z}_\infty \widetilde{\Gamma}^{[j]}$ are equal precisely when the class number $h$ is odd. For this reason, the bijection described in Lemma~\ref{lem:bijection} provides a canonical lifting of $f$ to an adelic automorphic form only when $h$ is odd.
\begin{prop}\label{prop:hodd}
    If $h$ is odd, then the $h$-tuple $(f_1,\ldots,f_h)$, where $f_j(g_\infty)=f(A_j^{-1} g_\infty)$, lifts to an adelic automorphic form $\Phi_f$ with trivial central character via the bijection described in Lemma~\ref{lem:bijection}.
\end{prop}
\begin{proof}
If $h$ is odd, then the map $[\frakm]\mapsto [\frakm]^2$ is a permutation on $\Cl_F$, hence the set \[\mathcal{D}=\{\frakm_1^2,\ldots,\frakm_h^2\}\] form a complete set of ideal class representatives for $\Cl_F$. Accordingly, let \[\widetilde{\Gamma}_{\mathcal{D}}^{[j]}=G(F)\cap t_j^2 (G_\infty K_f)\calZ_\A t_j^{-2}.\] By Lemmas~\ref{lem:Gammajcosetnew} and~\ref{lem:AjGammaj}, each $f_j$ is left-invariant under $\calZ_\infty \widetilde{\Gamma}_{\mathcal{D}}^{[j]}$. To use Lemma~\ref{lem:bijection}, we shall show that the tuple $(f_1,\ldots,f_h)$ is $\calZ$-compatible. Since every element of $\Cl_F$ is a square when $h$ is odd, it suffices to take $S=\{1\}$ and show that there exists $z\in \calZ_\A$, $\delta\in G(F)$, and $w=i_\infty(A_j^{-1}) w_f\in G_\infty K_f$ such that $z t_j^2=\delta w$. Write $A_j=\begin{pmatrix}
    a & b \\ c & d
\end{pmatrix}$. For each finite place $\frakp$, let $c'=c/\theta_j$, $d'=d/\theta_j$. Since $\langle c,d\rangle=\frakm_j$, and since $A_j$ is quasi-integral, the matrix $\begin{pmatrix}
    d' & -\theta_j b \\ -c' & \theta_j a
\end{pmatrix}$ is in $\GL_2(\calO_\frakp)$ for every finite place $\frakp$. Let $w=i_\infty(A_j^{-1})w_f$ with $w_f=\begin{pmatrix}
    d' & -\theta_j b \\ -c' & \theta_j a
\end{pmatrix}$. We have \[A_j  w=\left(\begin{matrix}
    \theta_j^{-1} &  \\  & \theta_j
\end{matrix}\right)=\left(\begin{matrix}
    \theta_j &  \\  & \theta_j
\end{matrix}\right) t_j^{2}.\]
This verifies that $(f_1,\ldots,f_h)$ is $\calZ$-compatible. Therefore, by Lemma~\ref{lem:bijection}, the tuple $(f_1,\ldots,f_h)$ lifts to an automorphic function $\Phi_f$ that is left $G(F)$-invariant, right $K_f$-invariant, and $\calZ_\A$-invariant. The cuspidality and moderate growth of the adelic lift follow directly from the corresponding properties of the original Maass cusp form $f$. Moreover, the $K_\infty$-invariance of $\Phi_f$ follows from the right $K_\infty$-invariance of $f$, and
the $Z\left(U(\mathfrak{g})\right)$-finiteness follows from the fact that $f$ is a Laplace eigenfunction, since the identification of the archimedean Casimir operator with the Laplacian on $\h^3$ preserves the eigenvalue. We conclude that $\Phi_f$ is an automorphic form.
\end{proof}

\begin{prop}\label{prop:heven}
Suppose $h$ is even. Let $1\leq j\leq h$, and let $1\leq k\leq h$ be the index such that $[\frakm_k]=[\frakm_j^2]$. Then $f_j$ can be written as a linear combination of the $k$-th components of the $h$-tuple of functions on $G_\infty$ projected from adelic automorphic forms whose central characters are Hecke characters lifted from class group characters.
\end{prop}
\begin{proof}
Let
\[\mathcal{D}=\{\frakn_1,\ldots,\frakn_h\},\]
where $\frakn_i=\frakm_i$ for $i\neq k$ and $\frakn_k=\frakm_j^2$. Accordingly, let
\[\Gamma_{\mathcal{D}}^{[i]}=\begin{cases} G(F)\cap t_i (G_\infty K_f) t_i^{-1} & \text{if} \quad i\neq k \\
 G(F)\cap t_j^2 (G_\infty K_f) t_j^{-2} &\text{if} \quad  i=k, \end{cases} \quad \widetilde{\Gamma}_{\mathcal{D}}^{[i]}=\begin{cases} G(F)\cap t_i (G_\infty K_f)\calZ_\A t_i^{-1} & \text{if} \quad i\neq k \\
 G(F)\cap t_j^2 (G_\infty K_f) \calZ_\A t_j^{-2} &\text{if} \quad  i=k. \end{cases}\]
    By Lemma~\ref{lem:AjGammaj}, $f_j$ is $\calZ_\infty\Gamma_{\mathcal{D}}^{[k]}$-invariant. Let $h'$ be the number of 2-torsion elements of $\Cl_F$, and let $\chi'_1,\ldots,\chi'_{h'}$ be the characters on $\Cl_F[2]$. For each $1\leq m\leq h'$, fix a character $\chi_m$ on $\Cl_F$ whose restriction to $\Cl_F[2]$ equals $\chi'_m$. By Lemma~\ref{lem:Gammatildejeigen}, we can write
\[f_j=\sum_{1\leq m\leq h'} f_{j,m},\]
where each $f_{j,m}$ transforms under $\calZ_\infty \widetilde{\Gamma}^{[k]}$ by the character $\chi_m$. For each $m$, we can construct an $(\calZ,\chi_m)$-compatible $h$-tuple $(\phi^{(1)}_{j,m},\ldots,\phi^{(h)}_{j,m})$ of functions on $G_\infty$ such that $\phi^{(k)}_{j,m}=f_{j,m}$ and each $\phi^{(i)}_{j,m}$ transforms under $\calZ_\infty \widetilde{\Gamma}_{\mathcal{D}}^{[i]}$ by the character $\chi_m$ as follows. We let $\phi^{(i)}_{j,m}$ be the function uniquely determined from $\phi^{(k)}_{j,m}$ by the $(\calZ,\chi_m)$-compatibility condition \eqref{eq:Zchicompatible} whenever the class $[\frakm_i]$ is a square in $\Cl_F$, and set $\phi^{(i)}_{j,m}\equiv 0$ otherwise. By Lemma~\ref{lem:chibijection}, this vector lifts to an adelic automorphic function $\Phi_{j,m}$ which is left $G(F)$-invariant, right $K_f$-invariant, and transforms under $\calZ_\A$ by the character $\omega_{\chi_m}$. Observe that the construction of $f_{j,m}$ given in the proof of Lemma~\ref{lem:Gammatildejeigen} preserves the right $K_\infty$-invariance and the Laplace eigenvalue of $f_j$. Also, each $\phi_{j,m}^{(i)}$ constructed using the $(\calZ,\chi_m)$-compatibility is right $K_\infty$-invariant and has the same Laplace eigenvalue. We conclude that $\Phi_{j,m}$ is right $K_\infty$-invariant and $Z\left(U(\mathfrak{g})\right)$-finite. Also, as before in Proposition~\ref{prop:hodd}, $\Phi_{j,m}$ inherits the cuspidality and the moderate growth property from $f$. Thus $\Phi_{j,m}$ is an adelic automorphic form, and since \[f_j=\sum_{1\leq m\leq h'} f_{j,m}=\sum_{1\leq m\leq h'} \phi^{(k)}_{j,m},\] the proposition follows.
\end{proof}
We now conclude this section with a precise adelic realization of the Maass cusp forms viewed at each cusp, using the results established above. Check that the following proposition immediately implies Theorem~\ref{thm:f_adelic}.
\begin{prop}\label{prop:f_adelic}
Let $f$ be an even Maass cusp form. Let $\mathcal{C}=\{\frakm_1,\ldots,\frakm_h\}$ be a set of ideal class representatives, $A_j$ be the matrix associated to $\frakm_j$ defined by \eqref{Ajdef}, and $\theta_j$ be the finite idele associated to $\frakm_j$. Then the function $f_j(z+r\bj)=f\bigl(A_j^{-1}(z+r\bj)\bigr)$ is a finite linear combination of the functions \begin{equation}\label{Phitj}\Phi\biggl(\begin{pmatrix}
    r & z \\ 0 & 1
\end{pmatrix}\begin{pmatrix}
    \theta_j^{-2} & 0 \\ 0 & 1
\end{pmatrix}\biggr),\end{equation}
where $\Phi$ ranges over adelic Hecke eigenforms that are left $G(F)$-invariant, right $K_\infty K_f$-invariant, and whose central characters are Hecke characters lifted from class group characters with trivial archimedean component.
\end{prop}
\begin{proof}
By Lemmas~\ref{lem:bijection} and~\ref{lem:chibijection} together with Propositions~\ref{prop:hodd} and~\ref{prop:heven}, $f_j$ is a finite linear combination of functions of the form \eqref{Phitj}, where $\Phi$ ranges over adelic automorphic forms with the stated properties. Each such adelic automorphic form is $Z\left(U(\mathfrak{g})\right)$-finite. By \cite[Theorem 1]{hc1968automorphic}, the space of right $K_\infty K_f$-invariant automorphic forms with a fixed Casimir eigenvalue is finite dimensional. Consequently, any such form can be expressed as a finite linear combination of Hecke eigenforms (see also \cite{ booker2016converse,borel2007automorphic}). This finishes the proof.
\end{proof}

\section{Bounds of $L$-functions}\label{sec:lfunctions}
In this section, we record the bounds of $L$-functions of $\GL_1$ and $\GL_2$ automorphic representations over the imaginary quadratic field $F$ which are used in estimating the integrals of the form \eqref{basicform}. \subsection{Bounds on Hecke $L$-functions} Let $\chi$ be a class group character. We extend $\chi$ to nonzero fractional ideals by $\chi(\frakm) = \chi([\frakm])$. The Hecke $L$-function associated to $\chi$ is defined by
\[
L(s,\chi) = \sum_{0 \neq \frakm \subseteq \calO_F}
\frac{\chi(\frakm)}{N(\frakm)^{s}},
\quad \real(s) > 1,
\]
which admits a meromorphic continuation to the entire complex plane.  
If $\chi$ is the trivial character, then $L(s,\chi)$ coincides with the Dedekind zeta function  $\zeta_F(s)$. Due to S\"{o}hne \cite{sohne1997hecke}, we have the following subconvexity bounds for the Hecke $L$-functions on the critical line.
\begin{lem}\label{lem:gl1subconv}
Let $\chi$ be a class group character. For any $\epsilon>0$, we have
\[L(1/2+it,\chi)\ll_\epsilon |t|^{\frac{1}{6}+\epsilon}.\]
\end{lem}
Also, by using the Vinogradov's method \cite[Chapter 6]{titchmarsh1986zeta} together with the previously known zero-free region of Hecke $L$-functions, we obtain the following bounds on $L(s,\chi)$ on the vertical line $\real (s)=1$. Note that the same approach is used in \cite{truelsen2011qe} for Hecke $L$-functions over totally real fields.
\begin{lem}\label{lem:Lft-bounds}
   There exists a positive number $L$ such that we have the estimates 
\begin{equation}\label{eq:bdlogLfunc}
        \frac{L'(1+it,\chi)}{L(1+it,\chi)}=O\Bigl((\log |t|)^{\frac{2}{3}}(\log \log |t|)^{\frac{1}{3}} \Bigr),
    \end{equation}
    and 
\begin{equation}\label{eq:bdreciLfunc}
         \frac{1}{L(1+it,\chi)}=O\Bigl((\log |t|)^{\frac{2}{3}}(\log \log |t|)^{\frac{1}{3}} \Bigr),
    \end{equation}
for $|t|\geq L$.
\end{lem}
\begin{proof}
By \cite[Satz 1.1]{hinz1980zerofree}, there is a sufficiently large $L$ and a constant $C_1$ such that $L(\delta+it,\chi)$ is zero-free in the region 
\[\delta\geq 1-\frac{C_1}{(\log |t|)^{\frac{2}{3}} (\log \log |t|)^{\frac{1}{3}}} \quad \text{and} \quad |t|\geq L. \]
Also, by \cite[Satz 2.1]{hinz1980zerofree}, there are constants $C_2$ and $C_3$ such that 
\[ L(\delta+it,\chi)=O \left(\exp(C_2 \log\log |t|)\right)\]
for 
\[\delta\geq 1-C_3\frac{(\log \log |t|)^{\frac{2}{3}}}{(\log |t|)^{\frac{2}{3}}} \quad \text{and} \quad |t|\geq L.\] 
Applying a classical result of Landau \cite[Theorem 3.10] {titchmarsh1986zeta}, following the approach in Proposition 3.3 of \cite{truelsen2011qe}, yields the desired bounds.
\end{proof}
\subsection{Bounds on $\GL_2$ automorphic $L$-functions} We follow the normalization of \cite[Section 4.6]{bump1997automorphic}. Let $\phi$ be a cuspidal automorphic form on $\mathrm{GL}_{2}(\mathbb{A}_{F})$ with Casimir eigenvalue $1+\nu^2$ and the central character $\omega$.
Assume that $\phi$ is right invariant under $K_{\infty}K_{f}$ and is a Hecke eigenform. 
For a prime ideal $\mathfrak{p}$ of $F$, let $\lambda(\mathfrak{p})$ be the Hecke eigenvalue of $\phi$ at $\mathfrak{p}$. 
Then the Satake parameters $\alpha_{\mathfrak{p}}, \beta_{\mathfrak{p}}$ satisfy 
\[\alpha_\frakp+\beta_\frakp= N(\mathfrak{p})^{-1/2} \lambda(\mathfrak{p}), \quad \alpha_\frakp \beta_\frakp = \omega(\mathfrak{p}).\]
For $\mathrm{Re}(s)>1$, the $L$-function $L(s,\phi)$ is defined by the Euler product
\[L(s,\phi)=\prod_{\frakp} \frac{1}{(1-\alpha_\frakp N(\mathfrak{p})^{-s})(1-\beta_\frakp N(\mathfrak{p})^{-s})}\]
and extends meromorphically to $\C$. If $\chi$ is a class group character, then the $L$-function twisted by $\chi$ is defined by
\begin{equation}\label{Lsphichipr}L(s,\phi\otimes\chi)=\prod_{\frakp} \frac{1}{(1-\alpha_\frakp\chi(\frakp) N(\mathfrak{p})^{-s})(1-\beta_\frakp\chi(\frakp) N(\mathfrak{p})^{-s})}.\end{equation}
The Hecke relations
\[\lambda(\frakp^{k+1})=\lambda(\frakp)\lambda(\frakp^k)- \omega (\frakp) N(\frakp) \lambda(\frakp^{k-1})\]
imply that the $L$-functions have the Dirichlet series representation
\[L(s,\phi)=\sum_{0\neq \frakm\subseteq \calO_F}\frac{\lambda(\frakm)}{N(\frakm)^{\frac{1}{2}+s}},\]
and
\begin{equation}\label{Lsphichi}L(s,\phi\otimes\chi)=\sum_{0\neq \frakm\subseteq \calO_F}\frac{\lambda(\frakm)\chi(\frakm)}{N(\frakm)^{\frac{1}{2}+s}}.\end{equation}
By \cite[Theorem 1.1]{michel2010subconv} of Michel and Venkatesh, we have the following subconvexity bound for $L(s,\phi\otimes \chi)$.
\begin{lem}\label{lem:gl2subconv}
There is an absolute constant $\delta>0$ such that 
\[
L(1/2+it,\phi\otimes\chi) \ll_{F} (1+|t|)^{1-\delta}.
\]
\end{lem}

\section{Quantum ergodicity of Eisenstein series} \label{sec:qe}
In this section, we prove our main result, Theorem~\ref{thm:qe}, using the explicit Fourier expansion of Eisenstein series and Maass cusp forms. The exact Fourier expansion of Eisenstein series at each cusp has been computed extensively in \cite{EGM2,heitkamp1992hecke}, and we recall the formula without modification. For the Maass cusp forms, we appeal to the adelic realization described in Section~\ref{sec:Maass}, and employ the Fourier expansions of Hecke eigenforms as given in \cite{assing2024supnorm,blomer2020supnorm}. Fix a set of cusp representatives
\begin{equation}\label{cuspset}\mathrm{C}_\Gamma=\{\eta_1,\ldots,\eta_h\}.\end{equation}
Let $A_j=A_{\eta_j}$ and $\frakm_j=\frakm_{\eta_j}$ be as defined in \eqref{Aetadef} and \eqref{frakmeta}. By the remark below \eqref{Ajdef}, this matches with the notation there. By \cite[Theorem 14.5]{heitkamp1992hecke} (see also \cite[Theorem 3.13]{raulf2014trace}), the Fourier expansion of the Eisenstein series $E_{\eta_j}$ at the cusp $\eta_i$ is given by
\begin{equation}\label{EetaFourier}
\begin{aligned}
E_{\eta_j}(A_i^{-1} v,s)&=\delta_{i,j} r^s+\tau_{i,j}(s) r^{2-s} +\frac{2(2\pi)^s}{h_F |d_F|^{s/2} \Gamma(s)} \frac{N(\frakm_i)^{2-s}}{N(\frakm_j)^s} \\ &\times \biggl(\sum_{0\neq n\in \frakm_i^2}\sum_{\chi\in \widehat{\Cl}_F}\frac{\chi(\frakm_i \frakm_j)|n|^{s-1}}{L(s,\chi)}\sigma_{1-s}\bigl(\chi,(n)\frakm_i^{-2}\bigr)r K_{s-1}\Bigl(\frac{4\pi|n|r}{\sqrt{|d_F|}}\Bigr)e\Bigl(\langle \frac{2\overline{n}}{\sqrt{d_F}},z\rangle\Bigr)\biggr),
\end{aligned}
\end{equation}
where \begin{equation}\label{tauij}\tau_{i,j}(s)=\frac{2\pi}{h_F (s-1) |d_F|^{1/2}} \frac{N(\frakm_i)^{2-s}}{N(\frakm_j)^{s}}\sum_{\chi\in\widehat{\Cl}_F}\chi(\frakm_i \frakm_j) \frac{L(s-1,\chi)}{L(s,\chi)}.\end{equation}
Here, $\delta_{i,j}$ is the delta symbol which equals $1$ if $i=j$ and $0$ otherwise, the bracket denotes the trace,
\[\langle z_1, z_2\rangle=\frac{z_1 \overline{z}_2+\overline{z}_1z_2}{2}=\real (z_1 \overline{z}_2),\]
and $\sigma_s(\chi,\frakm)$ denotes the divisor sum
\[\sigma_s(\chi,\frakm)=\sum_{\substack{\fraka\in \mathcal{M} \\ \frakm\subset\fraka\subset \calO_F}}\chi(\fraka) N(\fraka)^s.\]
Next, let $\phi$ be a Hecke eigenform with Casimir eigenvalue $1+\nu^2$, Hecke eigenvalues $\lambda(\frakm)$, and central character $\omega_\chi$, where $\omega_\chi$ is the Hecke character lifted from a class group character $\chi$ with trivial archimedean component. As in Section~\ref{sec:Maass}, let $\theta_j$ be the finite idele associated to $\frakm_j$. By \cite[(3.1)]{assing2024supnorm} and the notation from \cite[Section 8]{blomer2020supnorm}, the Whittaker expansion of $\phi$ in our setting equals
\[\begin{aligned}
    \phi\biggl(\begin{pmatrix}
        r & z \\ 0 & 1
    \end{pmatrix} \begin{pmatrix}
     \theta_j^{-1} & 0 \\ 0 & 1   
    \end{pmatrix} \biggr) 
    &= c_{\phi} \sum\limits_{q\in F^\times}\chi(\frakm_j)^{1/2} \frac{\lambda\bigl((q) \theta_j^{-1} \mathfrak{d}\bigr)}{N\bigl((q) \theta_j^{-1} \mathfrak{d}\bigr)}W_\infty \biggl(\begin{pmatrix}
        q & 0 \\ 0 & 1
    \end{pmatrix} \begin{pmatrix}
     r & z \\ 0 & 1   
    \end{pmatrix}\biggr) \\
    &= c_{\phi} \chi(\frakm_j)^{1/2} \sum\limits_{0\neq q\in (\theta_j^{-1} \mathfrak{d})^{-1}} \frac{\lambda\bigl((q) \theta_j^{-1} \mathfrak{d}\bigr)}{N\bigl((q) \theta_j^{-1} \mathfrak{d}\bigr)} \frac{|q r| K_{i \nu}\bigl(2\pi |2q| r\bigr)}{|\Gamma(1+i\nu)\Gamma(1-i\nu)|^{1/2}}e\bigl(\langle 2\overline{q},z\rangle\bigr).
\end{aligned}\]
Here the factor $\chi(\frakm_j)^{1/2}$ arises from the central character, whose archimedean component is trivial and hence is independent of $q$. We may regard the $\Gamma$-factors as a constant and absorb it in the constant $c_{\phi}$. Since $\mathfrak{d}=\sqrt{d_F} \calO_F$, we have
\begin{equation}\label{Heckeeigenfourier}\begin{aligned}
\phi\biggl(\begin{pmatrix}
        r & z \\ 0 & 1
    \end{pmatrix} \begin{pmatrix}
     \theta_j^{-1} & 0 \\ 0 & 1   
    \end{pmatrix} \biggr) 
    &= c_{\phi} \chi(\frakm_j)^{\frac{1}{2}}\sum\limits_{0\neq n\in \frakm_j}\frac{\lambda\bigl((n) \frakm_j^{-1} \bigr)}{N\bigl((n) \frakm_j^{-1}\bigr)} \frac{|n|}{\sqrt{|d_F|}} r K_{i\nu}\Bigl(\frac{4\pi |n| r}{\sqrt{|d_F|}}\Bigr) e\Bigl(\langle \frac{2\overline{n}}{\sqrt{d_F}},z\rangle\Bigr),
\end{aligned}\end{equation}
where $c_\phi$ is a constant depending on $\phi$. Note that while we fix representatives $\frakm_j$ and the corresponding finite ideles $\theta_j$, the Fourier expansion remains valid after replacing $\frakm_j$ by any fractional ideal in the same class and replacing $\theta_j$ accordingly.
\subsection{Contribution of Maass cusp forms}
The following proposition shows that the contribution of Maass cusp forms converges to $0$ as $t\to\infty$. 
\begin{prop}\label{prop:Maasscontribution}
    Let $f$ be a Maass cusp form. Then for any $1\leq j\leq h$ we have
    \[\lim_{t\to\infty}\int_{\Gamma\backslash \h^3} f(v) |E_{\eta_j}(v,1+it)|^2\frac{dx dy dr}{r^3}=0.\]
\end{prop}
\begin{proof}
    Let \[I_j(s)=\int_{\Gamma\backslash \h^3} f(v) E_{\eta_j}(v,1+it) E_{\eta_j}(s) \frac{dx dy dr}{r^3}.\]
Unfolding the integral using \eqref{Gammaetaconj}, we obtain
\[\begin{aligned}I_j(s)&=\int_{\Gamma\backslash \h^3} f(v) E_{\eta_j}(v,1+it) \sum_{\gamma \in \Gamma_{\eta_j}\backslash \Gamma}\Im(A_j \gamma v)^s \frac{dx dy dr}{r^3} \\
&=\int_{\Gamma_{\eta_j}\backslash \h^3} f(v) E_{\eta_j}(v,1+it) \Im(A_j v)^s \frac{dx dy dr}{r^3} \\ 
&=\int_{(A_j \Gamma_{\eta_j}A_j^{-1})\backslash \h^3}f(A_j ^{-1} v) E_{\eta_j}(A_j ^{-1} v,1+it) r^{s-3} dx dy dr \\
&=\int_{0}^{\infty}\int_{P(\frakm_j^{-2})}f(A_j ^{-1} v) E_{\eta_j}(A_j ^{-1} v,1+it) r^{s-3} dx dy dr,
\end{aligned}\]
where $P(\frakm)$ denotes the fundamental domain of the lattice $\frakm$. By Proposition~\ref{prop:f_adelic}, $f(A_j^{-1}v)$ is a finite linear combination of the functions of the form \eqref{Phitj}, with $\Phi$ a Hecke eigenform. Accordingly, consider a Hecke eigenform $\phi$ with central character $\omega_{\chi}$. By \eqref{Heckeeigenfourier}, we have
\[\phi\biggl(\begin{pmatrix}
        r & z \\ 0 & 1
    \end{pmatrix} \begin{pmatrix}
     \theta_j^{-2} & 0 \\ 0 & 1   
    \end{pmatrix} \biggr) 
    =C_j(\phi)\sum\limits_{0\neq n\in \frakm_j^2}\frac{\lambda\bigl((n) \frakm_j^{-2} \bigr)}{N\bigl((n) \frakm_j^{-2}\bigr)} \frac{|n|}{\sqrt{|d_F|}} r K_{i\nu}\Bigl(\frac{4\pi |n| r}{\sqrt{|d_F|}}\Bigr) e\Bigl(\langle \frac{2\overline{n}}{\sqrt{d_F}},z\rangle\Bigr),\]
    where $C_j(\phi)$ is a constant depending on $\phi$ and $j$. Let
\[I_{j,\phi}(s)=\int_{0}^{\infty}\int_{P(\frakm_j^{-2})}\phi\biggl(\begin{pmatrix}
        r & z \\ 0 & 1
    \end{pmatrix} \begin{pmatrix}
     \theta_j^{-2} & 0 \\ 0 & 1   
    \end{pmatrix} \biggr) E_{\eta_j}(A_j ^{-1} v,1+it) r^{s-3} dx dy dr.\]
To prove the proposition, it suffices to show that
\[\lim_{t\to \infty} I_{j,\phi}(1-it)=0.\]
By \eqref{EetaFourier} and the orthogonality, we obtain 
\[\begin{aligned}
  I_{j,\phi}(s) &=C_j(\phi)\mathrm{Vol}\bigl(P(\frakm_j^{-2})\bigr)\sum_{\chi_1\in\widehat{\Cl}_F}\int_0^\infty \frac{2(2\pi)^{1+it}N(\frakm_j)^{-2it}\chi_1(\frakm_j^2)}{h_F|d_F|^{1+\frac{it}{2}}\Gamma(1+it)L(1+it,\chi_1)} \\
  &\times \biggl(\sum_{0\neq n\in\frakm_j^2}\frac{|n|^{1+it}\sigma_{-it}\bigl(\chi_1,(n)\frakm_j^{-2}\bigr)\lambda\bigl((n)\frakm_j^{-2}\bigr)}{N\bigl((n)\frakm_j^{-2}\bigr)}K_{it}\Bigl(\frac{4\pi|n|r}{\sqrt{|d_F|}}\Bigr)K_{i\nu}\Bigl(\frac{4\pi|n|r}{\sqrt{|d_F|}}\Bigr)  \biggr)r^{s-1}dr \\
  &=C_j(\phi)\mathrm{Vol}\bigl(P(\frakm_j^{-2})\bigr)\frac{2(2\pi)^{1+it}N(\frakm_j)^{-2it}}{h_F|d_F|^{1+\frac{it}{2}}\Gamma(1+it)} (4\pi)^{-s}|d_F|^{\frac{s}{2}} \\
  &\times \biggl(\sum_{\chi_1\in\widehat{\Cl}_F}
  \frac{\chi_1(\frakm_j^2)}{L(1+it,\chi_1)}\sum_{0\neq n\in\frakm_j^2}\frac{\sigma_{-it}\bigl(\chi_1,(n)\frakm_j^{-2}\bigr)\lambda\bigl((n)\frakm_j^{-2}\bigr)}{N\bigl((n)\frakm_j^{-2}\bigr) |n|^{s-it-1} }  \biggr)\int_{0}^{\infty} K_{it}(r)K_{i\nu}(r) r^{s-1}dr.
\end{aligned}\]
By \cite[6.576]{gradshteyn2015table}, we have
\begin{equation}\label{besselidentity}\int_{0}^{\infty}K_{it}(r)K_{i\nu}(r)r^s\frac{dr}{r}= \frac{\Gamma(\frac{s+it+i\nu}{2})\Gamma(\frac{s+it-i\nu}{2})\Gamma(\frac{s-it+i\nu}{2})\Gamma(\frac{s-it-i\nu}{2})}{\Gamma(s)}.\end{equation}
Denote gamma factors on the right hand side by $T(s)$. We have
\begin{equation}\label{Ijphi}I_{j,\phi}(s)=C_j(\phi)\mathrm{Vol}(P(\frakm_j^{-2}))\frac{2^{1-s}(2\pi)^{1+it-s}N(\frakm_j)^{-s-it+1}}{h_F|d_F|^{1+\frac{it}{2}-\frac{s}{2}}\Gamma(1+it) }T(s) \sum_{\chi_1\in \widehat{\Cl}_F}\frac{\chi_1(\frakm_j^2)R_{\chi_1}(s)}{L(1+it,\chi_1)}, \end{equation}
with 
\[R_{\chi_1}(s)=N(\frakm_j)^{s-it-1}\sum_{0\neq n\in \frakm_j^2} \frac{\sigma_{-it}\bigl(\chi_1,(n)\frakm_j^{-2}\bigr)\lambda\bigl((n)\frakm_j^{-2}\bigr)}{N\bigl((n)\frakm_j^{-2}\bigr) |n|^{s-it-1} }=\sum_{0\neq n\in \frakm_j^2} \frac{\sigma_{-it}\bigl(\chi_1,(n)\frakm_j^{-2}\bigr)\lambda\bigl((n)\frakm_j^{-2}\bigr)}{N\bigl((n)\frakm_j^{-2}\bigr)^{\frac{s-it+1}{2}} }
.\]
Now, we use the averaging technique to write $R_{\chi_1}(s)$ in terms of $L$-functions. For any fractional ideal $\frakm$, we have
\begin{equation}\label{eq:charsum}\frac{1}{h_F}\sum_{\chi\in \widehat{\Cl}_F} \chi(\frakm_j^2\frakm)=\begin{cases}
    1 & \text{if } [\frakm]=[\frakm_j^{-2}] \\
    0 & \text{otherwise.}
\end{cases}\end{equation}
Thus
\[R_{\chi_1}(s)=\frac{1}{h_F}\sum_{\chi_2\in \widehat{\Cl}_F}\chi_2(\frakm_j^2)\sum_{0\neq \frakm\subseteq \calO_F} \frac{\sigma_{-it}(\chi_1,\frakm)\lambda(\frakm)\chi_2(\frakm)}{N(\frakm)^{\frac{1}{2}+\frac{s-it}{2}}}. \]
Denote the inner sum by $R_{\chi_1,\chi_2}(s)$. Using the multiplicativity of the divisor sum, Hecke eigenvalue, and the class group character together with \eqref{Lsphichipr} and \eqref{Lsphichi}, we obtain
\[\begin{aligned}
R_{\chi_1,\chi_2}(s)&=\prod_{\frakp \text{ prime}} \sum_{k=0}^{\infty} \frac{\sigma_{-it}(\chi_1,\frakp^k)\lambda(\frakp^k)\chi_2(\frakp^k)}{N(\frakp)^{k(\frac{s-it}{2}+\frac{1}{2})}} \\
&=\prod_{\frakp \text{ prime}} \sum_{k=0}^{\infty} \frac{\lambda(\frakp^k)\chi_2(\frakp^k)}{N(\frakp)^{k(\frac{s-it}{2}+\frac{1}{2})}}\sum_{l=0}^k \chi_1(\frakp)^lN(\frakp)^{-itl} \\ 
&=\prod_{\frakp \text{ prime}} \sum_{k=0}^{\infty} \frac{\lambda(\frakp^k)\chi_2(\frakp^k)}{N(\frakp)^{k(\frac{s-it}{2}+\frac{1}{2})}} \frac{1-\chi_1(\frakp)^{k+1}N(\frakp)^{-it(k+1)}}{1-\chi_1(\frakp)N(\frakp)^{-it}} \\
&=\prod_{\frakp \text{ prime}} \frac{1}{1-\chi_1(\frakp)N(\frakp)^{-it}} \Bigl(\sum_{k=0}^{\infty} \frac{\lambda(\frakp^k)\chi_2(\frakp^k)}{N(\frakp)^{k(\frac{s-it}{2}+\frac{1}{2})}}-\chi_1(\frakp)N(\frakp)^{-it}\sum_{k=0}^{\infty}\frac{\lambda(\frakp^k)\chi_1(\frakp^k)\chi_2(\frakp^k)}{N(\frakp)^{k(\frac{s+it}{2}+\frac{1}{2})}}\Bigr) \\
&= \prod_{\frakp \text{ prime}}\frac{1}{1-\chi_1(\frakp)N(\frakp)^{-it}}\Bigl(\frac{1}{(1-\alpha_\frakp \chi_2(\frakp)N(\frakp)^{-\frac{s-it}{2}})(1-\beta_\frakp \chi_2(\frakp)N(\frakp)^{-\frac{s-it}{2}})} \\ 
&\quad\quad\quad\quad\quad\quad\quad\quad\quad\quad\quad\quad-\frac{\chi_1(\frakp)N(\frakp)^{-it}}{(1-\alpha_\frakp \chi_1(\frakp)\chi_2(\frakp)N(\frakp)^{-\frac{s+it}{2}})(1-\beta_\frakp \chi(\frakp)\chi_2(\frakp)N(\frakp)^{-\frac{s+it}{2}})}\Bigr) \\
&=\prod_{\frakp \text{ prime}}(1-\alpha_\frakp \beta_\frakp \chi_1(\frakp)\chi_2(\frakp)^2 N(\frakp)^{-s} )\Bigl(\frac{1}{(1-\alpha_\frakp \chi_2(\frakp)N(\frakp)^{-\frac{s-it}{2}})(1-\beta_\frakp \chi_2(\frakp)N(\frakp)^{-\frac{s-it}{2}})}\Bigr) \\
&\quad\quad\times \Bigl(\frac{1}{(1-\alpha_\frakp \chi_1(\frakp)\chi_2(\frakp)N(\frakp)^{-\frac{s+it}{2}})(1-\beta_\frakp \chi_1(\frakp)\chi_2(\frakp)N(\frakp)^{-\frac{s+it}{2}})}\Bigr)\\
&=\frac{L(\frac{s-it}{2},\phi \otimes \chi_2)L(\frac{s+it}{2},\phi \otimes \chi_1\chi_2)}{L(s,\chi \chi_1 \chi_2^2)}.
\end{aligned}\]
Thus
\[R_{\chi_1,\chi_2}(1-it)=\frac{L(\frac{1}{2}-it,\phi \otimes \chi_2)L(\frac{1}{2},\phi \otimes \chi_1\chi_2)}{L(1-it,\chi \chi_1\chi_2^2)}. \]
By \eqref{eq:bdreciLfunc} and Lemma~\ref{lem:gl2subconv}, we have 
\[\frac{R_{\chi_1,\chi_2}(1-it)}{L(1+it,\chi_1)}=O(|t|^{1-\delta})\] for some $\delta>0$. On the other hand, by Stirling's formula
\[|\Gamma(\sigma+it)|\sim e^{-\frac{\pi|t|}{2}} |t|^{\sigma-\frac{1}{2}},\]
we have
\[\frac{T(1-it)}{\Gamma(1+it)}=O(|t|^{-1}).\]
By \eqref{Ijphi}, we conclude that
\[\lim_{t\to\infty} I_{j,\phi}(1-it)=0.\]
This finishes the proof.
\end{proof}

\subsection{Contribution of incomplete Eisenstein series}
Let $\psi(y)$ be a rapidly decreasing function on $\R^+$, so $\psi(y)=O_N(y^N)$ as $y\to\infty$ or $y\to 0$ for any $N\in\Z$. Then its Mellin transform
\[\Psi(s)=\int_0^\infty \psi(y) y^{-s} \frac{dy}{y}\]
is entire, and is of Schwartz class in $t$ for each vertical line $\sigma+it$. We have
\[\psi(y)=\frac{1}{2\pi i}\int_{(\sigma)} \Psi(s) y^s ds\]
for all $\sigma\in\R$. For such $\psi$ we define the incomplete Eisenstein series associated to the cusp $\eta_i$ by
\begin{equation}\label{Eetaincomplete}E_{\eta_i}(v|\psi)=\sum_{\gamma\in \Gamma_{\eta_i} \backslash \Gamma} \psi\bigl(\Im(A_i\gamma v)\bigr)=\frac{1}{2\pi i}\int_{(3)} \Psi(s) E_{\eta_i} (v,s)ds.\end{equation}
Let $\xi(s,\chi)$ denote the completed Hecke $L$-function,
\begin{equation}\label{xidef}\xi(s,\chi)=2(2\pi)^{-s}|d_F|^{\frac{s}{2}}\Gamma(s)L(s,\chi).\end{equation}
Then $\tau_{i,j}(s)$ in \eqref{tauij} satisfies
\begin{equation}\label{tauij2}\tau_{i,j}(s)=\frac{N(\frakm_i)^{2-s}}{N(\frakm_j)^s}\Bigl(\frac{1}{h_F}\sum_{\chi\in\widehat{\Cl}_F}\chi(\frakm_i\frakm_j)\frac{\xi(s-1,\chi)}{\xi(s,\chi)}\Bigr),\end{equation}
and the Fourier expansion of $E_{\eta_j}(A_i^{-1} v,s)$ in \eqref{EetaFourier} equals
\begin{equation}\label{EetaFourier2}
    \begin{aligned}
&\delta_{i,j} r^s+\tau_{i,j} (s) r^{2-s} +\sum_{0\neq n\in \frakm_i^2} \omega_{i,j}(n,s) |n|^{s-1}r K_{s-1}\Bigl(\frac{4\pi|n|r}{\sqrt{|d_F|}}\Bigr)e\Bigl(\langle \frac{2\overline{n}}{\sqrt{d_F}},z\rangle\Bigr),
\end{aligned}
\end{equation}
where
\begin{equation}\label{omegaij}\omega_{i,j}(n,s)=\frac{4N(\frakm_i)^{2-s}}{N(\frakm_j)^s} \Bigl(\frac{1}{h_F}\sum_{\chi\in \widehat{\Cl}_F}\frac{\chi(\frakm_i \frakm_j)}{\xi(s,\chi)}\sigma_{1-s}\bigl(\chi,(n)\frakm_i^{-2}\bigr)\Bigr).\end{equation}
The following proposition computes the contribution of the incomplete Eisenstein series.
\begin{prop}\label{prop:Eisensteincontribution}
For any $1\leq j\leq h$, we have
\[\int_{\Gamma\backslash\h^3} E_{\eta_i}(v|\psi) |E_{\eta_j}(v,1+it)|^2 dV = \frac{(2\pi)^2 N(\frakm_j^{-2})}{|d_F| \zeta_F(2) } \Bigl(
\int_{\Gamma\backslash\h^3} E_{\eta_i}(v|\psi) dV \Bigr) \log t + O\Bigl((\log t)^{\frac{1}{3}}(\log\log t)^{\frac{2}{3}}\Bigr) \]
as $t\to\infty$.
\end{prop}
\begin{proof}
Write $I_{i,j}(t)$ for the first integral in the proposition. Writing the incomplete Eisenstein series as the rightmost expression in \eqref{Eetaincomplete} and unfolding, we obtain
\[\begin{aligned}
    I_{i,j}(t)&=\frac{1}{2\pi i} \int_0^\infty \int_{P(\frakm_i^{-2})}\int_{(3)}  \Psi(s) |E_{\eta_j}(A_i^{-1}v,1+it)|^2 r^{s-3} ds \,dx\,dy \,dr.
\end{aligned}\]
Using the Fourier expansion \eqref{EetaFourier2} and the orthogonality of the Fourier modes, we obtain
\[\begin{aligned}
 I_{i,j}(t)&=\frac{\mathrm{Vol}\bigl(P(\frakm_i^{-2})\bigr)}{2\pi i} \int_0^\infty \int_{(3)} \Psi(s)r^s \biggl[ \Bigl( \delta_{i,j} r^{1+it} +
\tau_{ij}(1+it) r^{1-it}\Bigr) \Bigl(\delta_{i,j} r^{1-it} + \tau_{ij}(1-it) r^{1+it} \Bigr) \\
&\quad\quad\quad\quad\quad\quad\quad+ \sum_{0 \neq n \in \frakm_i^2} \omega_{i,j}(n,1+it)\omega_{i,j}(n,1-it) r^2 \Bigl| K_{it}\Bigl(\frac{4\pi |n|r}{\sqrt{|d_F|}}\Bigr)\Bigr|^2 \biggr]  ds \frac{dr}{r^3}.
\end{aligned}\]
Write $I_{i,j}(t)=F^{(1)}_{i,j}(t)+F^{(2)}_{i,j}(t)$, where $F^{(1)}_{i,j}(t)$ is the constant-term contribution 
\[\frac{\mathrm{Vol}\bigl(P(\frakm_i^{-2})\bigr)}{2\pi i} \int_0^\infty \int_{(3)} \Psi(s)r^s  \Bigl(\delta_{i,j} r^{1+it} +
\tau_{ij}(1+it) r^{1-it}\Bigr) \Bigl(\delta_{i,j}r^{1-it} + \tau_{ij}(1-it) r^{1+it} \Bigr) ds \frac{dr}{r^3}.\]
By the functional equation of $\xi(s,\chi)$, we have
\[|\xi(1+ it,\chi)|=|\xi(-it,\chi^{-1})|=|\xi(it,\chi)|,\]
and the same holds if we replace $it$ with $-it$. By \eqref{tauij2}, we deduce that $\tau_{ij}(1\pm it)=O(1)$. Thus
\[\begin{aligned}
    F^{(1)}_{i,j}(t)&= \mathrm{Vol}\bigl(P(\frakm_i^{-2})\bigr) \int_0^\infty \psi(r)   \Bigl( \delta_{i,j} r^{1+it} +
\tau_{ij}(1+it) r^{1-it}\Bigr) \Bigl(\delta_{i,j} r^{1-it} + \tau_{ij}(1-it) r^{1+it} \Bigr) \frac{dr}{r^3} \\ 
&\ll \int_0^\infty \psi(r)\frac{dr}{r}.
\end{aligned}\]
We conclude that $F^{(1)}_{i,j}(t)=O(1)$. \par 
For the contribution of the non-constant terms, we compute
\[\begin{aligned}
F^{(2)}_{i,j}(t)&=\frac{\mathrm{Vol}(P(\frakm_i^{-2}))}{2\pi i} \int_{(3)} \Psi(s) \biggl[ \sum_{0 \neq n \in \frakm_i^2} \omega_{i,j}(n,1+it)\omega_{i,j}(n,1-it) \int_0^\infty \Bigl| K_{it}\Bigl(\frac{4\pi |n|r}{\sqrt{|d_F|}}\Bigr)\Bigr|^2 r^s\,\frac{dr}{r}\biggr] ds \\
&=\frac{\mathrm{Vol}(P(\frakm_i^{-2}))}{2\pi i} \int_{(3)} \Psi(s) \Bigl(\frac{4\pi }{\sqrt{|d_F|}}\Bigr)^{-s}  \frac{\Gamma(\frac{s}{2})^2\Gamma(\frac{s}{2}+it)\Gamma(\frac{s}{2}-it)}{\Gamma(s)} \\ 
&\quad\quad\times \Bigl(\sum_{0 \neq n \in \frakm_i^2} \frac{\omega_{i,j}(n,1+it)\omega_{i,j}(n,1-it)}{|n|^s}\Bigr) ds,
\end{aligned}\]
by using \eqref{besselidentity} for the Bessel integral. By \eqref{omegaij}, the last summation above equals
\[\frac{16N(\frakm_i)^2}{N(\frakm_j)^2 h_F^2}\sum_{\chi_1,\chi_2\in\widehat{\Cl}_F}\frac{(\chi_1\chi_2)(\frakm_i\frakm_j)}{\xi(1+it,\chi_1)\xi(1-it,\chi_2)} \sum_{0 \neq n \in \frakm_i^2} \frac{\sigma_{-it}\bigl(\chi_1,(n)\frakm_i^{-2}\bigr)\sigma_{it}\bigl(\chi_2,(n)\frakm_i^{-2}\bigr)}{|n|^s}.\]
Similarly as before, we use \eqref{eq:charsum} and rewrite the inner sum as a sum over integral ideals as follows.
\[\begin{aligned}
 &\sum_{0 \neq n \in \frakm_i^2} \frac{\sigma_{-it}\bigl({\chi_1},(n)\frakm_i^{-2}\bigr)\sigma_{it}\bigl({\chi_2},(n)\frakm_i^{-2}\bigr)}{|n|^s} \\ &=\frac{N(\frakm_i)^{-s}}{h_F}\sum_{{\chi_3}\in\widehat{\Cl}_F} {\chi_3}(\frakm_i^2)\Bigl(\sum_{0 \neq \frakm \subseteq \calO_F}\frac{\sigma_{-it}({\chi_1},\frakm)\sigma_{it}({\chi_2},\frakm){\chi_3}(\frakm)}{N(\frakm)^{\frac{s}{2}}}\Bigr).\end{aligned}\]
By multiplicativity, we have
 \[\begin{aligned}
 &\sum_{0 \neq \frakm \subseteq \calO_F}\frac{\sigma_{-it}({\chi_1},\frakm)\sigma_{it}({\chi_2},\frakm){\chi_3}(\frakm)}{N(\frakm)^{\frac{s}{2}}}=\prod_{\frakp \text{ prime}}\sum_{k=0}^{\infty}\frac{\sigma_{-it}({\chi_1},\frakp^k)\sigma_{it}({\chi_2},\frakp^k){\chi_3}(\frakp^k)}{N(\frakp^k)^{\frac{s}{2}}} \\
 &=\prod_{\frakp \text{ prime}}\sum_{k=0}^{\infty} \frac{{\chi_3}(\frakp^k)}{N(\frakp)^{\frac{ks}{2}}}\Bigl(\frac{1-{\chi_1}(\frakp)^{k+1}N(\frakp)^{-it(k+1)}}{1-{\chi_1}(\frakp)N(\frakp)^{-it}}\Bigr)\Bigl(\frac{1-{\chi_2}(\frakp)^{k+1}N(\frakp)^{it(k+1)}}{1-{\chi_2}(\frakp)N(\frakp)^{it}}\Bigr) \\
 &=\prod_{\frakp \text{ prime}}\frac{1}{(1-{\chi_1}(\frakp)N(\frakp)^{-it})(1-{\chi_2}(\frakp)N(\frakp)^{it})}\Bigl(\frac{1}{1-{\chi_3}(\frakp)N(\frakp)^{-\frac{s}{2}}}+\frac{{\chi_1}(\frakp){\chi_2}(\frakp)}{1-({\chi_1}{\chi_2}{\chi_3})(\frakp)N(\frakp)^{-\frac{s}{2}}} \\
 &\quad\quad\quad\quad-\frac{{\chi_1}(\frakp)N(\frakp)^{-it}}{1-({\chi_1}{\chi_3})(\frakp)N(\frakp)^{-(\frac{s}{2}+it)}}-\frac{{\chi_2}(\frakp)N(\frakp)^{it}}{1-({\chi_2}{\chi_3})(\frakp)N(\frakp)^{-(\frac{s}{2}-it)}}\Bigr) \\
 &=\prod_{\frakp \text{ prime}}\frac{1-{\chi_3}(\frakp)^2{\chi_1}(\frakp){\chi_2}(\frakp)N(\frakp)^{-s}}{(1-{\chi_3}(\frakp)N(\frakp)^{-\frac{s}{2}})(1-({\chi_1}{\chi_2}{\chi_3})(\frakp)N(\frakp)^{-\frac{s}{2}})} \\
 &\quad\quad\quad\quad\times \frac{1}{(1-({\chi_1}{\chi_3})(\frakp)N(\frakp)^{-(\frac{s}{2}+it)})(1-({\chi_2}{\chi_3})(\frakp)N(\frakp)^{-(\frac{s}{2}-it)})} \\
 &=\frac{L(\frac{s}{2},{\chi_3})L(\frac{s}{2},{\chi_1}{\chi_2}{\chi_3})L(\frac{s}{2}+it,{\chi_1}{\chi_3})L(\frac{s}{2}-it,{\chi_2}{\chi_3})}{L(s,{\chi_1}{\chi_2}{\chi_3}^2)}.
\end{aligned}\]
We deduce that
\begin{equation}\label{Fij2t}\begin{aligned}
 &F^{(2)}_{i,j}(t)=\frac{8\mathrm{Vol}\bigl(P(\frakm_i^{-2})\bigr)N(\frakm_i)^2}{\pi i N(\frakm_j)^2 h_F^3 } \sum_{{\chi_1},{\chi_2},{\chi_3} \in\widehat{\Cl}_F}\frac{({\chi_1}{\chi_2})(\frakm_i\frakm_j){\chi_3}(\frakm_i)^2}{\xi(1+it,{\chi_1})\xi(1-it,{\chi_2})} \int_{(3)}B_{{\chi_1},{\chi_2},{\chi_3}}(s)ds,
\end{aligned}\end{equation}
where
\[\begin{aligned}
   B_{{\chi_1},{\chi_2},{\chi_3}}(s)&=\Psi(s)\Bigl(\frac{4\pi N(\frakm_i)}{\sqrt{|d_F|}}\Bigr)^{-s} \frac{\Gamma(\frac{s}{2})^2\Gamma(\frac{s}{2}+it)\Gamma(\frac{s}{2}-it)}{\Gamma(s)} \\
    &\times \frac{L(\frac{s}{2},{\chi_3})L(\frac{s}{2},{\chi_1}{\chi_2}{\chi_3})L(\frac{s}{2}+it,{\chi_1}{\chi_3})L(\frac{s}{2}-it,{\chi_2}{\chi_3})}{L(s,{\chi_1}{\chi_2}{\chi_3}^2)}.
\end{aligned}\]
Note that the only possible poles of $B_{\chi_1,\chi_2,\chi_3}(s)$ are at $s=2$ and $s=2\pm 2it$. By Stirling's formula and the fact that $\Psi(\sigma+iT)$ is rapidly decreasing in $T$, we can make a contour shift and write
\[\begin{aligned}
    &\int_{(3)}B_{{\chi_1},{\chi_2},{\chi_3}}(s)\,ds =2\pi i\left(\mathrm{Res}_{s=2}B_{{\chi_1},{\chi_2},{\chi_3}}(s)+\mathrm{Res}_{s=2\pm 2it}B_{{\chi_1},{\chi_2},{\chi_3}}(s)\right)+\int_{(1)}B_{{\chi_1},{\chi_2},{\chi_3}}(s)\,ds.
\end{aligned}\]
By Stirling's formula and Lemma~\ref{lem:gl1subconv}, we have
\[\frac{1}{\xi(1+it, {\chi_1})\xi(1-it, {\chi_2})} \int_{(1)}B_{{\chi_1},{\chi_2},{\chi_3}}(s)\,ds\ll t^{-\frac{2}{3}+\epsilon}.\]
It remains to compute the residues. Observe that $B_{{\chi_1},{\chi_2},{\chi_3}}(s)$ has a simple pole at $s=2+2it$ if $\chi_3^{-1}=\chi_2$, and at $s=2-2it$ if $\chi_3^{-1}=\chi_1$, but in either case the residue is negligible because $\Psi(2+2it)$ decays rapidly as $t\to\infty$. We turn to the residue at $s=2$. Note that $B_{{\chi_1},{\chi_2},{\chi_3}}(s)$ has a double pole at $s=2$ if both ${\chi_3}$ and ${\chi_1}{\chi_2}{\chi_3}$ are trivial, a simple pole if precisely one of them is trivial, and has no pole otherwise. If there is no pole at $s=2$, then the contribution at the residue is $0$. Next, suppose that ${\chi_3}$ is trivial and ${\chi_1}{\chi_2}$ is not. Then we have
\[\begin{aligned} \frac{\mathrm{Res}_{s=2}B_{{\chi_1},{\chi_2},{\chi_3}}(s)}{\xi(1+it, {\chi_1})\xi(1-it, {\chi_2})}&=C_1 \frac{\Gamma(1+it)\Gamma(1-it)L(1+it,{\chi_1})L(1-it,{\chi_2})}{\xi(1+it, {\chi_1})\xi(1-it, {\chi_2})}=C_2,\end{aligned}\]
where $C_1$ and $C_2$ are independent of $t$. A similar computation shows that the contribution of the residue at $s=2$ for ${\chi_3}$ nontrivial and ${\chi_1}{\chi_2}{\chi_3}$ trivial is also $O(1)$. \par
Finally, suppose that both ${\chi_3}$ and ${\chi_1}{\chi_2}{\chi_3}$ are trivial characters, so that $B_{{\chi_1},{\chi_2},{\chi_3}}$ has a double pole at $s=2$. In this case, we have
\[\begin{aligned}
    B_{{\chi_1},{\chi_2},{\chi_3}}(s)&=\Psi(s)\Bigl(\frac{4\pi N(\frakm_i)}{\sqrt{|d_F|}}\Bigr)^{-s} \frac{\Gamma(\frac{s}{2})^2\Gamma(\frac{s}{2}+it)\Gamma(\frac{s}{2}-it)}{\Gamma(s)} \frac{\zeta_F(\frac{s}{2})^2L(\frac{s}{2}+it,{\chi_1})L(\frac{s}{2}-it,{\chi_2})}{\zeta_F(s)}.
\end{aligned}\]
Write $B_{{\chi_1},{\chi_2},{\chi_3}}(s)=\zeta_F(\frac{s}{2})^2 G_{{\chi_1},{\chi_2},{\chi_3}}(s)$, where $G_{{\chi_1},{\chi_2},{\chi_3}}(s)$ is holomorphic at $s=2$. 
By the class number formula, we have the Laurent expression
\[\zeta_F\Bigl(\frac{s}{2}\Bigr)=\frac{A_{-1}}{s-2}+A_0+O(s-2)\]
with
\[A_{-1}=\frac{2\pi h_F}{ \sqrt{|d_F|}}.\]
Then
\[B_{{\chi_1},{\chi_2},{\chi_3}}(s)=\Bigl(\frac{A_{-1}}{s-2}+A_0+O(s-2)\Bigr)^2\left(G_{{\chi_1},{\chi_2},{\chi_3}}(2)+G_{{\chi_1},{\chi_2},{\chi_3}}'(2)(s-2)+O((s-2)^2)\right)\]
near $s=2$, so
\[\mathrm{Res}_{s=2} B_{{\chi_1},{\chi_2},{\chi_3}}(s)=\lim_{s\to 2} \frac{d}{ds}\Bigl((s-2)^2 B_{{\chi_1},{\chi_2},{\chi_3}}(s)\Bigr)=G_{{\chi_1},{\chi_2},{\chi_3}}(2)A_{-1}\Bigl(2A_0+A_{-1}\frac{G_{{\chi_1},{\chi_2},{\chi_3}}'}{G_{{\chi_1},{\chi_2},{\chi_3}}}(2)\Bigr).\]
We have
\[\begin{aligned}
    G_{{\chi_1},{\chi_2},{\chi_3}}(2)&=\frac{L(1+it,{\chi_1})L(1-it,{\chi_2})}{\zeta_F(2)}\Bigl(\frac{4\pi N(\frakm_i)}{\sqrt{|d_F|}}\Bigr)^{-2} \Gamma(1+it)\Gamma(1-it) \Psi(2) \\ 
    &=\frac{\xi(1+it,{\chi_1})\xi(1-it,{\chi_2})}{16N(\frakm_i)^2\zeta_F(2) }\Psi(2),
\end{aligned}\]
and
\[\begin{aligned}
    \frac{G_{{\chi_1},{\chi_2},{\chi_3}}'}{G_{{\chi_1},{\chi_2},{\chi_3}}}(2)=\frac{\Psi'}{\Psi}(2)+\frac{L'(1+it,{\chi_1})}{2L(1+it,{\chi_1})}+\frac{L'(1-it,{\chi_2})}{2L(1-it,{\chi_2})}+\frac{\Gamma'(1+it)}{2\Gamma(1+it)}+\frac{\Gamma'(1-it)}{2\Gamma(1-it)}+C,
\end{aligned}\]
where $C$ is independent of $t$. By \eqref{eq:bdlogLfunc} and Stirling's formula, we deduce
\[\begin{aligned}\mathrm{Res}_{s=2} B_{{\chi_1},{\chi_2},{\chi_3}}(s)&=G_{{\chi_1},{\chi_2},{\chi_3}}(2)A_{-1}^2\Bigl(\log t+O\Bigl((\log t)^{\frac{1}{3}}(\log\log t)^{\frac{2}{3}}\Bigr)\Bigr) \\
&=\frac{(2\pi)^2h_F^2\xi(1+it,{\chi_1})\xi(1-it,{\chi_2})}{16|d_F| N(\frakm_i)^2\zeta_F(2)}\Psi(2)\Bigl(\log t+O\Bigl((\log t)^{\frac{1}{3}}(\log\log t)^{\frac{2}{3}}\Bigr)\Bigr).
\end{aligned}\]
Inserting this into \eqref{Fij2t} and using the fact that $I_{i,j}(t)=F_{i,j}^{(1)}(t)+F_{i,j}^{(2)}(t)$, we obtain
\[\begin{aligned} I_{i,j}(t)&=\frac{1}{ N(\frakm_j)^2 }  \frac{(2\pi)^2}{\zeta_F(2)|d_F|} \mathrm{Vol}\bigl(P(\frakm_i^{-2})\bigr) \Psi(2) \log t+O\Bigl((\log t)^{\frac{1}{3}}(\log\log t)^{\frac{2}{3}}\Bigr).
\end{aligned}\]
Finally, check that
\[\Psi(2)= \int_{0}^\infty \psi(r) \frac{dr}{r^3}=\frac{1}{\mathrm{Vol}\bigl(P(\frakm_i^{-2})\bigr)}\int_{\Gamma\backslash\h^3} E_{\eta_i}(v|\psi) \frac{dx dy dr}{r^3}.\]
From this, the proposition follows.
\end{proof}
We are now ready to prove Theorem~\ref{thm:qe}.
\begin{proof}[Proof of Theorem~\ref{thm:qe}]
    We follow the argument in the proof of \cite[Proposition 2.3]{luosarnak1995qe}. 
    For a function $f$ on $\Gamma\backslash \mathbb{H}^{3}$, we define
\[ A_t(f)=\frac{1}{\log t}\int_{\Gamma\backslash \mathbb{H}^{3}} f(v)|E_{\eta_j}(v,1+it)|^2 \frac{dx dy dr}{r^3} -\frac{(2\pi)^2 N(\frakm_j^{-2})}{\zeta_F(2)|d_F|} \Bigl(
\int_{\Gamma\backslash\h^3} f(v) \frac{dx dy dr}{r^3} \Bigr). \]
    Since the characteristic function of a compact Jordan measurable set can be approximated by continuous compactly supported functions, it is enough to show that 
   \[ \lim_{t\to \infty}A_t(f)=0 \]
   for every continuous function $f$ with compact support. 
   Note that the space of all incomplete Eisenstein series and cusp forms is dense in the space of continuous functions vanishing at all cusps. 
   Thus, given any continuous compactly supported function $f$ and $\epsilon>0$, there exists
$g=g_1+g_2$ such that $g_1$ is a finite linear combination of cusp forms, $g_2$ lies in the
space of incomplete Eisenstein series, and $\|f-g\|_{\infty}<\epsilon$. Let $H=f-g$. Since $H$ is rapidly decreasing at each cusp, there exists a nonnegative function
   \[ H_1(v)=\sum_{i=1}^{h} E_{\eta_i}(v|\psi_i), \]
   with each $\psi_i$ nonnegative and rapidly decreasing on $\R^+$, such that $H_1(v)\geq |H(v)|$ for all $v\in\h^3$, and
   \[ \int_{\Gamma\backslash \mathbb{H}^{3}} H_1(v)\frac{dx dy dr}{r^3}\ll_{F}\epsilon. \]
By Proposition~\ref{prop:Eisensteincontribution}, we have
   \[\begin{aligned}
        \limsup_{t\to\infty} |A_t(H)| &\leq \limsup_{t\to\infty}\Bigl(\frac{1}{\log t}\int_{\Gamma\backslash \mathbb{H}^{3}} H_1(v)|E_{\eta_j}(v,1+it)|^2 dV\Bigr)  +  \frac{(2\pi)^2 N(\frakm_j^{-2})}{\zeta_F(2)|d_F|} \Bigl(
\int_{\Gamma\backslash\h^3} H_1(v) dV \Bigr) \\
           &\ll_{F} \epsilon.
       \end{aligned}
   \]
   Also, since $g_1$ integrates to zero over $\Gamma\backslash\h^3$, Proposition~\ref{prop:Maasscontribution} implies $\lim_{t\to\infty}A_t(g_1)=0$. Finally, by Proposition~\ref{prop:Eisensteincontribution}, we have $\lim_{t\to\infty}A_t(g_2)=0$.
Therefore
\begin{equation*}
   \limsup_{t\to\infty} |A_t(f)|\leq \limsup_{t\to\infty} \Bigl(|A_t(H)|+|A_t(g_1)|+|A_t(g_2)|\Bigr) \ll_{F} \epsilon.
\end{equation*}
Since the choice of $\epsilon$ is arbitrary, this yields the desired result $\lim_{t\to\infty}A_t(f)=0$.
\end{proof}

\bibliographystyle{plain}
\bibliography{qe_eisseries.bib}
\end{document}